\documentclass[12pt]{amsart}
\usepackage{amsmath}
\usepackage{amscd}
\usepackage{amssymb}
\usepackage{amsfonts}
\usepackage{graphicx}
\usepackage{wrapfig}
\usepackage{amsthm}
\usepackage{euscript}
\usepackage[utf8]{inputenc}
\usepackage[english]{babel}
\usepackage{pgf,tikz}
\usepackage[colorlinks, citecolor=blue]{hyperref}

\usepackage{comment}
\excludecomment{hidden}

\def\RR{\mathbb{R}}

\def\SS{\mathbb{S}}

\def\rd{\overset{\circ}{Ric}}
\def\R{\overset{\circ}R}

\def\<{\langle}
\def\>{\rangle}

\newtheorem{theorem}{Theorem}[section]
\newtheorem{lemma}[theorem]{Lemma}
\newtheorem{corollary}[theorem]{Corollary}
\newtheorem{proposition}[theorem]{Proposition}

\theoremstyle{definition}

\theoremstyle{remark}
\newtheorem{remark}[theorem]{Remark}

\numberwithin{equation}{section}

\usepackage[utf8]{inputenc}
\usepackage{tikz}
\usetikzlibrary{patterns}
\pgfdeclarepatternformonly{my dots}{\pgfqpoint{-1pt}{-1pt}}{\pgfqpoint{2.5pt}{2.5pt}}{\pgfqpoint{6pt}{6pt}}%
{
	\pgfpathcircle{\pgfqpoint{0pt}{0pt}}{.5pt}
	\pgfpathcircle{\pgfqpoint{7pt}{7pt}}{.5pt}
	\pgfusepath{fill}
}
\begin{document}

\title[Four-dimensional shrinkers]
{Four-dimensional shrinkers with pinched curvatures}

\author{Xiaodong Cao}
\address{Department of Mathematics, Cornell University, Ithaca, NY 14853}
\email{xiaodongcao@cornell.edu}

\author{Jia-Yong Wu}
\address{Department of Mathematics and Newtouch Center for Mathematics, Shanghai University, Shanghai 200444, China}
\email{wujiayong@shu.edu.cn}

\subjclass[2010]{Primary 53C25; Secondary 53C20, 53C21}
\dedicatory{}
\date{\today}
\keywords{Gradient shrinking Ricci solition; Ricci flow; four-manifolds; Einstein manifold;
pinching curvature; biorthogonal curvature; Hitchin-Thorpe inequality}

\begin{abstract}
In this paper, we first classify complete four-dimensional shrinkers under a pinching condition on the $2$nd-Ricci curvature. We then classify closed four-dimensional shrinkers with (weighted) integral parameter-pinching of the half Weyl curvature under several pinching assumptions on the biorthogonal curvature. Finally, we establish a Hitchin-Thorpe type inequality for closed non-trivial four-dimensional
shrinkers under some assumptions on the half Weyl curvature and the biorthogonal curvature.
\end{abstract}
\maketitle

\section{Introduction}\label{Int1}
Let $(M,g)$ be an $n$-dimensional complete Riemannian manifold and let $f$
be a smooth function on $M$. The triple $(M, g, f)$ is called a
\emph{gradient shrinking Ricci soliton} or simply a \emph{shrinker}
(see \cite{[Ham]}), if there exists a constant $\lambda>0$ such that
\begin{equation}\label{Eq1}
Ric+\mathrm{Hess}\,f=\lambda  g,
\end{equation}
where $Ric$ denotes the Ricci curvature and $\text{Hess}\,f$ denotes the Hessian of $f$ on $(M,g)$. When $f$ is constant, the shrinker is referred to as trivial. Shrinkers are self-similar solutions to the Ricci flow and play a central role in analyzing singularity formation
of the Ricci flow \cite{[Ham]}. We refer readers to the comprehensive surveys \cite{[Cao1],[Cao2]} and the references therein for background on Ricci solitons.

Extensive classification results for shrinkers have been established over the past decades.
The full classification in dimensions two and three is already complete, as documented in
\cite{[Ham],[CCZ],Iv,[NW],[Pe2]}. In dimension four and higher, the classification problem
remains largely open, yet substantial progress has been achieved by imposing various
curvature pinching restrictions; see, e.g., \cite{CaCh,CJZ,CaXi,CaJT,CWZ,Cat13,CWy,CheZh,LNW,[MS],[MuW17],Nab,[PW],WWW,Zhz}
and the references therein. More recently, all K\"ahler shrinking surfaces have been
fully classified in \cite{LW23} by combining earlier works \cite{BCCD,CCD,CDS}.

In this paper, we focus on rigidity and topological theorems for four-dimensional shrinkers
under some new curvature pinching hypotheses. We briefly recall major prior rigidity results
to contextualize our works.
Cao \cite{Ca07} proved that any compact shrinker with positive curvature operator 
is Einstein if it satisfies an integral inequality. B\"ohm-Wilking \cite{BW} applied the 
Ricci flow to prove that every closed shrinker with two-positive curvature operators must 
be a quotient of
the standard $n$-sphere $\mathbb{S}^n$. Brendle-Schoen  \cite{bs} used Ricci flow
arguments to show that closed shrinkers with weakly $\frac14$-pinched sectional curvature
are isometric to $\SS^4$, $\mathbb{RP}^4$ or $\mathbb{CP}^2$ equipped with their canonical
metrics. Catino \cite{[Cat]} adopted elliptic PDE methods to obtain a weaker variant of
the Brendle-Schoen rigidity theorem, which states that every simply, connected, oriented,
complete four-dimensional shrinker satisfying suitable sectional curvature pinching
condition is isometric to the standard ${\RR}^4$, $\SS^4$ or $\mathbb{CP}^2$.
In addition, more rigidity results for closed shrinkers have been obtained under
various pinching curvature conditions; see, e.g., \cite{CaTr16,CaJW,Cat16,Wj}
and the references therein.

Our first result delivers a classification of four-dimensional shrinkers under a
pinching condition on the $2$nd-Ricci curvature, generalizing  Catino's result under
the pinching condition on the sectional curvature \cite{[Cat]}. Recall that the
$2$nd-Ricci curvature on $(M,g)$ satisfies $Ric_2(x)\geq H$ for some constant
$H=H(x)$ at a point $x\in M$ if
\[
Rm(e_1, e_2, e_1, e_2)+Rm(e_1, e_3, e_1, e_3)\ge H
\]
for any three orthonormal vectors $\{e_1, e_2, e_3\}$ in $T_xM$, where $Rm$ denotes the Riemannian curvature tensor. If this inequality holds at all points $x\in M$, we say that the $2$nd-Ricci curvature
of $(M,g)$ satisfies $Ric_2(M)\ge H$ for some $H\in C^\infty(M)$. The $2$nd-Ricci curvature can be regarded as an interpolation between the Ricci curvature and the sectional curvature. For more details, we refer the reader to \cite{Shen,Wh} and the references therein.

\begin{theorem}\label{Rig1}
Let $(M,g,f)$ be a simply connected, oriented, complete four-dimensional shrinker. If
\begin{equation}\label{pinicon}
Ric_2(M)\ge2\varepsilon R
\end{equation}
for some $\varepsilon>\frac{1}{24}$, where $R$ denotes the scalar curvature of $(M,g)$,
then $(M,g,f)$ is isometric to either $\mathbb{R}^4$, $\SS^4$, or $\mathbb{CP}^2$ with its
standard metric.
\end{theorem}

We state some results related to Theorem \ref{Rig1}. On the one hand, for a simply connected four-dimensional Einstein manifold $(M,g)$ with $Ric=\lambda\, g$ for some constant $\lambda>0$,
if the sectional curvature satisfies $K\le\frac{5}{6}\lambda$, then by Corollary 1.3 in \cite{CaTr22} (or
a weaker version in Corollary 1.3 of \cite{CuSu}), $(M,g)$ must be isometric to either $\SS^4$
or $\mathbb{CP}^2$ with its standard metric. Note that $Ric_2(M)\ge\frac{1}{6}\,\lambda(=\frac{1}{24}R)$ implies $K\le\frac{5}{6}\lambda$
(see Proposition \ref{relat} below). Therefore, if the Einstein manifold $(M,g)$ satisfies
\begin{equation}\label{rigcond}
Ric_2(M)\ge\frac{1}{24}R,
\end{equation}
then Einstein manifold $(M,g)$ must be isometric to either $\SS^4$ or $\mathbb{CP}^2$ \
with its standard metric.
On the other hand, Li and Ni \cite{LN} showed that any complete shrinker with
$Ric_2(M)>0$ is closed (see also \cite{WW}). Accordingly, Theorem \ref{Rig1}
may be regarded as a synthesis of the Einstein manifold classification from
\cite{CaTr22,CuSu} and the compactness criterion of shrinkers proven in \cite{LN}.

The proof of Theorem \ref{Rig1} is inspired by the work of Catino \cite{[Cat]}.
Our argument proceeds in two steps. First, we demonstrate that any shrinker
satisfying the pinching condition \eqref{pinicon} is necessarily Einstein.
This step relies on the Weitzenb\"ock formula for the squared norm of the traceless
Ricci curvature, combined with several delicate algebraic curvature estimates arising
from the lower bound on the $2$nd-Ricci curvature. Second, for the Einstein manifold
satisfying \eqref{pinicon}, we directly apply Corollary 1.3 in \cite{CaTr22} to
complete the classification.

Our second objective is to classify closed four-dimensional shrinkers under several
pinching assumptions on the biorthogonal curvature. Recall that in \cite{Sea,Noro},
at each point $x\in M$, the biorthogonal (sectional) curvature of a plane
$P\subset T_x(M)$ is defined by
\[
K^{\perp}(x, P)=\frac{K(x, P)+K(x, P^{\perp})}{2},
\]
where $P^\perp$ denotes the orthogonal complement of $P$ and $K$ is the sectional curvature.
For each point $x\in M$, the minimum and maximum biorthogonal curvatures are
defined by
\[
K^\perp_{min}(x)=\min_{P\subset T_xM}\, K^{\perp}(x, P)\quad\text{and}\quad
K^\perp_{max}(x)=\max_{P\subset T_xM}\, K^{\perp}(x, P),
\]
respectively. We often denote $K^{\perp}(P)$, $K^\perp_{min}$, $K^\perp_{max}$, and
omit the base point $x$ for simplicity. It is known that a four-dimensional manifold
is Einstein if and only if  $K^{\perp}=K$, i.e. $K^{\perp}(P)=K(P)$ for every plane $P\subset T_xM$ at each point $x\in M$. A four-dimensional manifold $M$ is locally conformally flat if and only if $K^{\perp}=\frac{1}{12}R$.  Clearly, positive sectional curvature implies positive biorthogonal curvature. However, positive biorthogonal curvature does not even imply
positive Ricci curvature; see e.g., $\SS^3\times\SS^1$ with its standard metric.
Indeed, positive biorthogonal curvature can be viewed as an intermediate condition
between positive sectional curvature and positive scalar curvature. According to definitions of $K^{\perp}$ and $Ric_2(M)$, two curvatures
do not imply each other.

Inspired by \cite{CaTr22}, we prove some rigidity theorems for closed four-dimensional shrinkers with
integral parameter-pinching of the harmonic half Weyl curvature under several pinching assumptions
on the biorthogonal curvature.
\begin{theorem}\label{Rigid2}
Let $(M,g,f)$ be a simply connected, oriented, closed four-dimensional shrinker
satisfying
\begin{equation}\label{bddeW}
\int_M|\delta W^{\pm}|^2dv\le k\int_M R|W^{\pm}|^2dv
\end{equation}
for some $0\le k\le1/4$, where $W^{\pm}$ denotes the half Weyl curvature and $\delta$ is the divergence operator. Let $\lambda_1(M)$ be the first eigenvalue of the Laplacian on functions. If one of the following conditions holds:
\begin{eqnarray}\label{list}
&K^{\perp}\le\frac{2(1-4k)(1-2k)R^2+(3-8k)\lambda_1(M)R}{12[(1-4k)R+\lambda_1(M)]},\label{la}\\
&K^\perp\ge\frac{(1-16k^2)R^2+8kR\lambda_1(M)}{24[(1-4k)R+\lambda_1(M)]},\label{lb}\\
&K^\perp_{min}\ge\frac{(1-16k^2)R+8k\lambda_1(M)}{4(1-4k)(1-2k)R+2(3-8k)\lambda_1(M)}K^\perp_{max},\label{lc}\\
& Ric\ge a \,\,\, \text{and}\,\,\, K^\perp\le\frac{2(1-4k)(1-2k)+1-\tfrac{8}{3}k}{4(1-3k)}a \,\,\, \text{for some constant}\,\,\, a>0, \label{1d}
\end{eqnarray}
then $(M,g,f)$  is isometric to either $\SS^4$ or $\mathbb{CP}^2$ with its standard metric.
\end{theorem}
\begin{remark}\label{remK2}
We make some comments on the pinching conditions in Theorem \ref{Rigid2}.
\begin{itemize}
\item [(1)] When $k=1/4$, condition \eqref{la} simplifies to $K^{\perp}\le\frac{1}{12}R$, and
our result is not new. Indeed, by \cite{CoRi,Sea}, if a manifold satisfies
$K^{\perp}\le\frac{1}{6}R$, then it naturally has nonnegative isotropic
curvature, which immediately yields a full classification of shrinkers by \cite{LNW}.

\item [(2)] On the standard $\SS^4$ or $\mathbb{CP}^2$, one has $\delta W^{\pm}=0$
 and $K^{\perp}=\frac{1}{12}R$. For $k=1/4$, conditions \eqref{bddeW} and \eqref{lb}
reduce to
\[
\int_M|\delta W^{\pm}|^2dv\le \frac{1}{4}\int_M R|W^{\pm}|^2dv
\quad \mathrm{and}\quad K^{\perp}\ge\frac{1}{12}R.
\]
In this sense, Theorem \ref{Rigid2} can be
interpreted as a gap theorem characterizing these standard spaces.

\item [(3)] Again taking $k=1/4$, condition \eqref{lb} becomes $K^{\perp}\ge\frac{1}{12}R$.
This condition is neither stronger nor weaker than the curvature pinching condition
\eqref{pinicon} in Theorem \ref{Rig1}, which can be easily verified by inspecting 
the definitions of $K^{\perp}$and $Ric_2(M)$. From this comparison, we conclude that 
Theorem \ref{Rigid2} and Theorem \ref{Rig1} do not imply each other.

\item [(4)] The case $k=0$ forces $\delta W^{\pm}=0$ from \eqref{bddeW}, and the shrinker is Einstein
by \cite{WWW}. For Einstein manifolds, our biorthogonal curvature pinching bound
is stricter than the condition established by Cao and Tran (see Corollary 1.3 in \cite{CaTr22}). Consequently, our rigidity statement recovers no new results when $k=0$.

\item [(5)] Define a continuous decreasing function
\[
U(k):=\frac{2(1-4k)(1-2k)R^2+(3-8k)\lambda_1(M)R}{12[(1-4k)R+\lambda_1(M)]}\quad
\text{on $[0,1/4)$}.
\]
Taking $k\to 0+$, condition \eqref{la} turns into
\[
K^{\perp}\le U(0+)=\frac{2R^2+3\lambda_1(M)R}{12(R+\lambda_1(M))}.
\]
A direct computation indicates that $U(0+)>\frac{1}{6}R$. By continuity of $U(k)$,
there exists some $k_1\in(0,1/4)$, for example, we can take 
$k_1$ satisfying 
\[
0<k_1<\frac{R+2\lambda_1(M)-\sqrt{R^2+4\lambda_1(M)^2}}{8R},
\]
such that $U(k_1)>\frac{1}{6}R$;
the pinching bound $K^{\perp}\le U(k_1)$ is therefore weaker than $K^{\perp}\le\frac{1}{6}R$.
To our knowledge, our rigidity result is new in this parameter regime.
\end{itemize}
\end{remark}

The proof strategy of Theorem \ref{Rigid2} is inspired by \cite{Yang,CaTr22}.
We first prove that condition \eqref{la} implies the conclusion of the theorem.
The idea of the proof is as follows. If the half Weyl curvature vanishes, the conclusion
follows from Chen-Wang's classification \cite{CWy}. Otherwise, we derive a contradiction
by estimating the integral in Cao-Tran's Weitzenb\"ock formula \cite{CaTr16},
combined with the Kato inequality, the Poincar\'e inequality of the first eigenvalue,
and some algebraic curvature estimates \cite{CaTr22}. The other cases of Theorem
\ref{Rigid2} then follow from the fact that either \eqref{lb}, \eqref{lc}, or \eqref{1d}
implies \eqref{la} (see Lemma \ref{recond}). We note that Cao-Tran's pinching condition
on the biorthogonal curvature \cite{CaTr22} (see also \cite{DRF,Rib}) is weaker than ours,
because their proof uses the refined Kato inequality for harmonic
two-forms or harmonic Weyl curvature on four-dimensional manifolds, which
is not applicable in our setting.

By an argument similar to that of Theorem \ref{Rigid2}, we give another classification
of closed four-dimensional shrinkers with a weighted integral parameter-pinching condition on
the half Weyl curvature under an upper bound on the biorthogonal curvature.

\begin{theorem}\label{Rigid3}
Let $(M,g,f)$ be a simply connected, oriented, closed four-dimensional shrinker
satisfying \eqref{Eq1}. Assume that
\begin{equation}\label{piweyl}
\int_M\langle Ric\circ Ric, W^{\pm}\rangle e^{-f}dv\le 4k\int_M |W^{\pm}|^2e^{-f}dv
\end{equation}
for some $0\le k\le\lambda$. Let $\lambda_{1,f}(M)$ be the first eigenvalue of the $f$-Laplacian
on functions. If
\begin{equation}\label{weigh}
K^\perp\le\frac{(4\lambda-4k)(4\lambda-4k+R)+\lambda_{1,f}(M)(8\lambda-8k+R)}{12(4\lambda-4k+\lambda_{1,f}(M))},
\end{equation}
then $(M,g,f)$  is isometric to either $\SS^{4}$ or $\mathbb{CP}^2$ with its standard metric.
\end{theorem}
\begin{remark}\label{wrem}
We make some comments on condition  \eqref{weigh} in Theorem \ref{Rigid3}.
\begin{itemize}
\item [(1)]
From \cite{ChZ17}, we know that $\lambda_{1,f}(M)>\lambda$ on any compact shrinker.
Consider the function
\[
F(\lambda_{1,f}(M))=\frac{(4\lambda-4k)(4\lambda-4k+R)+\lambda_{1,f}(M)(8\lambda-8k+R)}{12(4\lambda-4k+\lambda_{1,f}(M))}.
\]
This expression is nondecreasing as a function of $\lambda_{1,f}(M)$. Consequently, the
following inequality
\[
K^\perp\le \frac{(4\lambda-4k)(4\lambda-4k+R)+\lambda(8\lambda-8k+R)}{12(4\lambda-4k+\lambda)},
\]
implies condition \eqref{weigh}, which in turn guarantees the rigidity conclusion of Theorem \ref{Rigid3}.

\item [(2)]
When $k=\lambda$, condition \eqref{weigh} simplifies to $K^{\perp}\le\frac{1}{12}R$.
As noted in Remark \ref{remK2} (1), this pinching bound recovers a well-known
classification result and therefore yields no new geometric information.

\item [(3)] In the special case $k=0$, we cannot deduce $\delta W^{\pm}=0$
solely from the integral assumption \eqref{piweyl}. This suggests that
our result is novel for this parameter regime. Specifically, whenever
\[
4\lambda(4\lambda-R)+\lambda_{1,f}(M)(8\lambda-R)>0,
\]
condition \eqref{weigh} is weaker than $K^\perp\le\frac{1}{6}R$,
so our classification yields a new conclusion within this framework.
\end{itemize}
\end{remark}

Finally, we turn our attention to topological properties of closed four-dimensional shrinkers.
It is a classical fact that every closed four-dimensional Einstein manifold obeys the
\emph{Hitchin-Thorpe inequality}
\[
|\tau(M)|\le\frac{2}{3}\chi(M),
\]
where $\chi(M)$ and $\tau(M)$ denote the Euler number the signature of $M$, respectively.
This inequality provides a fundamental topological obstruction to the existence of
Einstein metrics on closed four-dimensional manifolds.

Since shrinkers are generalizations of Einstein manifolds, and non-trivial
four-dimensional shrinkers on connected sums
$\mathbb{CP}^2\sharp\, k(-\mathbb{CP}^2)$ $(k=1,2)$ satisfy the classical
Hitchin-Thorpe inequality, H.-D. Cao \cite{[Cao1]} raised a natural question:
does an analogous Hitchin-Thorpe type inequality hold for arbitrary closed
four-dimensional shrinkers? This problem has been explored under various
curvature assumptions in prior work; see \cite{CRZ,FG,Ma,Tad} and the references
therein. Most recently, Cao, Ribeiro Jr and Wondo \cite{CaJW} proved a
Hitchin-Thorpe type inequality under lower bounds on modified sectional and
scalar curvatures. Motivated by their strategy, we derive a new $(\varepsilon,\sigma)$-pinching
quantitative condition on biorthogonal curvature for such
an inequality to hold on non-trivial closed four-dimensional shrinkers.

\begin{theorem}\label{HT}
Let $(M,g,f)$ be an oriented, non-trivial, closed four-dimensional shrinker with
\eqref{Eq1} satisfying
\begin{equation}\label{dWpin}
\int_M|\delta W^{\pm}|^2dv\le \frac 16\int_M R|W^{\pm}|^2dv.
\end{equation}
If $\varepsilon\lambda\le K^{\perp}\le\sigma\lambda$ for some constants
$\varepsilon\le \sigma\le 1/2$, then
\[
|\tau(M)|<\left(\frac{2}{3}-\frac{22}{27}\cdot\frac{1-2\sigma}{12\varepsilon^2-8\varepsilon+2\sigma+1}\right)\chi(M)
\]
provided the coefficient of $\chi(M)$ is positive.
\end{theorem}
\begin{remark}
Since the scalar curvature $R\ge 0$ (\cite{[Chen]}) and its relation to the biorthogonal curvature in \eqref{scalarandKperp} below,  the lower bound assumption on the biorthogonal curvature is automatically satisfied on closed four-dimensional shrinkers. We impose this lower bound assumption to potentially derive a sharp Hitchin–Thorpe type inequality.
In particular, only letting 
$\sigma=1/2$, Theorem \ref{HT} gives 
\[
|\tau(M)|<\frac23\chi(M).
\]	
\end{remark}

The proof of Theorem \ref{HT} modifies the proof trick of \cite{CaJW},
by combining the Gauss-Bonnet-Chern formula, the classical Hirzebrush's theorem
and the $L^2$-integral of half Weyl curvature relating with topological invariants
\cite{CaTr16}. In our proof, the bounds imposed on the biorthogonal curvature yield
uniform control over both the scalar curvature and the sum of the norms of the
self-dual and anti-self-dual Weyl curvatures. Assembling all these algebraic and
integral estimates together, we arrive at the desired topological inequality.

The paper is organized as follows. In Section \ref{pre}, we review curvature
decompositions and algebraic curvature inequalities for four-dimensional manifolds.
We also review some basic equations and properties about shrinkers. In Section
\ref{prorig1}, we study some algebraic curvature estimates under the pinching
condition on the $2$nd-Ricci curvature. Then we apply them to prove Theorem \ref{Rig1}.
In Section \ref{prorig2}, we investigate the relationship between lower and upper
bounds of biorthogonal curvature and prove Theorem \ref{Rigid2}. In Section
\ref{discuss}, we adapt the argument developed for Theorem \ref{Rigid2} to
prove Theorem \ref{Rigid3}. Finally, Section \ref{HTI} contains the proof of Theorem
\ref{HT}.

\section{Preliminaries}\label{pre}
In this section, we recall several known curvature decomposition results, the biorthogonal curvature, and their algebraic curvature estimates for four-dimensional manifolds. We also recall some basic equations and propositions concerning shrinkers. These results will be useful in the proofs of our theorems.

\subsection{Curvature decomposition}
Let $(M,g)$ be an oriented complete $n$-dimensional ($n\ge3$) Riemannian manifold.
The Riemannian curvature can be regarded as an adjoint operator on the vector
bundle of 2-forms, denoted by $\wedge^2 TM$. Owing to the algebraic structure of this
vector bundle, the Riemannian curvature $Rm$ admits the decomposition
\begin{equation}\label{curde}
Rm=W+\frac{R}{2n(n-1)}g\circ g+\frac{1}{n-2}\left(Ric-\frac{R}{n}Id\right)\circ g,
\end{equation}
where $Ric$, $R$, $W$ denote the Ricci curvature, the scalar curvature and the
Weyl curvature, respectively. For symmetric $2$-tensors $A$ and $B$, their
Kulkarni-Nomizu product $\circ$ is defined by
\[
(A\circ B) (u, v, z, w):= A(u,z)B(v,w)+B(u,z)A(v,w)-A(u,w)B(v,z)-B(u,w)A(v,z)
\]
for any vector fields $u, v, z, w\in TM$.

On an oriented four-dimensional manifold $(M,g)$, the Hodge star operator
$*:\wedge^2TM\to \wedge^2TM$ yields a natural decomposition of the vector
bundle of 2-forms $\wedge^2 TM$:
\begin{equation*}
\wedge^2 TM =\wedge^+ M \oplus \wedge^- M,
\end{equation*}
where $\wedge^{\pm} M$ denote the eigenspaces corresponding to eigenvalues $\pm 1$
of the Hodge star operator $*$, respectively. Elements of $\wedge^+ M$ and $\wedge^- M$
are called self-dual and anti-self-dual 2-forms, respectively. Since the Riemannian
curvature can be viewed as an operator acting on the space of 2-forms
$Rm: \wedge^2 TM\to \wedge^2 TM$, the Hodge star induces the block decomposition
\begin{equation*}
Rm=\left( \begin{array}{cc}
{W^+}+\frac{R}{12}Id & \frac{1}{2}(Ric-\frac{R}{4}Id)\circ g \\
\frac{1}{2}(Ric-\frac{R}{4}Id)\circ g & {W^-}+ \frac{R}{12}Id
\end{array} \right),
\end{equation*}
where the self-dual and anti-self-dual Weyl curvatures $W^{\pm}$ (also known as
 half Weyl curvatures) are the restriction of $W$ to self-dual and anti-self-dual
2-forms $\wedge^{\pm}M$, respectively.

Since the Weyl curvature $W$ is traceless and obeys the first Bianchi identity,
a beautiful decomposition of $W$ is available due to Berger \cite{ber}.

\begin{proposition}\label{berger}
Let $(M,g)$ be a four-dimensional manifold. At each point $x\in M$, there exists
an orthonormal basis $\{e_i\}_{1\le i\le 4}$ of $T_x M$ such that $W$ takes the form
\begin{equation*}
W=\left( \begin{array}{cc}
A & B\\
B & A
\end{array}\right),
\end{equation*}
with respect to the associated basis $\{e_i\wedge e_j\}_{1\le i<j\le 4}$
of $\wedge^2 T_xM$), where $A=\mathrm{Diag}\{a_1,a_2,a_3\}$, $B=\mathrm{Diag}\{b_1,b_2,b_3\}$.
Moreover, the following properties hold:
\begin{enumerate}
\item $a_1=W(e_1,e_2,e_1,e_2)=W(e_3,e_4,e_3,e_4)=\min_{|a|=|b|=1,~ a\perp b}W(a,b,a,b)$,
\item $a_3=W(e_1,e_4,e_1,e_4)=W(e_1,e_4,e_1,e_4)=\max_{|a|=|b|=1,~ a\perp b}W(a,b,a,b)$,
\item $a_2=W(e_1,e_3,e_1,e_3)=W(e_2,e_4,e_2,e_4)$,
\item $b_1=W_{1234}, b_2=W_{1342}, b_3=W_{1423}$,
\item $a_1+a_2+a_3=b_1+b_2+b_3=0$,
\item $|b_2-b_1|\le a_2-a_1, |b_3-b_1|\le a_3-a_1, |b_3-b_2|\le a_3-a_2$.
\end{enumerate}
\end{proposition}
Furthermore, we may construct orthonormal bases for $\wedge^{\pm} M$ as
\[
(w^{\pm}_1,w^{\pm}_2,w^{\pm}_3)=\frac{1}{\sqrt{2}}(e_{12}\pm e_{34}, e_{13}\mp e_{24}, e_{14}\pm e_{23}),
\]
where $e_{ij}:=e_i\wedge e_j$. Consequently, the ordered eigenvalues on
$\wedge^{\pm} M$ are given by
\begin{equation}\label{eigen}
\lambda_1^{\pm}=a_1\pm b_1\le\lambda_2^{\pm}=a_2\pm b_2\le\lambda_3^{\pm}=a_3\pm b_3.
\end{equation}
On a four-dimensional manifold $(M,g)$, recall the definition of biorthogonal curvature.
Suppose $P$ and $P^\perp$ are spanned by orthonomal pairs $\{e_1, e_2\}$ and $\{e_3, e_4\}$,
respectively. Then, in view of \eqref{curde}, we obtain
\begin{equation}
\begin{aligned}\label{orthogonalsec}
K^\perp(P)=\frac{1}{2}\big(Rm(e_1, e_2, e_1, e_2)+Rm(e_3, e_4, e_3, e_4)\big)
=W_{1212}+\frac{R}{12}.
\end{aligned}
\end{equation}
For three planes $P_1, P_2, P_3$ spanned by $\{e_1, e_2\}$, $\{e_1, e_3\}$,
$\{e_1, e_4\}$, respectively, we have
\begin{equation}\label{scalarandKperp}
K^\perp(P_1)+K^\perp(P_2)+K^\perp(P_3)=\frac{R}{4}.
\end{equation}

Using Proposition \ref{berger} together with purely algebraic computations,
the following two lemmas are established in \cite{CaTr22}. These estimates
will be used in the proof of Theorem \ref{Rigid2}.
\begin{lemma}\label{ineqlem}
Let $(M,g)$ be a four-dimensional manifold. Then
\begin{align*}
&\lambda_3^{\pm}\le\lambda_3^{+}+\lambda_3^{-}\le 2K^\perp_{max}-\frac{R}{6},\\
&|\lambda_3^+-\lambda_3^-|\le 2\left(K^\perp_{max}-\frac{R}{12}\right).
\end{align*}
Moreover, equality holds in both inequalities above if and only if $|W^-||W^+|=0$.
\end{lemma}

\begin{lemma}\label{detWup}
Let $(M,g)$ be a four-dimensional manifold. Then
\[
6\,\mathrm{det}\,W^{\pm}\le\lambda^{\pm}_3|W^\pm|^2.
\]
\end{lemma}

On the other hand, from $\lambda_1^{\pm}\le\lambda_2^{\pm}\le\lambda_3^{\pm}$
and $\lambda_1^{\pm}+\lambda_2^{\pm}+\lambda_3^{\pm}=0$, we claim that
$|W^{\pm}|^2\le 6(\lambda_1^{\pm})^2$. Indeed,
\begin{align*}
|W^{\pm}|^2&=\lambda_1^{\pm}+\lambda_2^{\pm}+\lambda_3^{\pm}\\
&=2(\lambda_1^{\pm})^2-2\lambda_2^{\pm}\lambda_3^{\pm}\\
&\le2(\lambda_1^{\pm})^2-2\lambda_1^{\pm}\lambda_3^{\pm}\\
&=2(\lambda_1^{\pm})^2-2\lambda_1^{\pm}(-\lambda_1^{\pm}-\lambda_2^{\pm})\\
&=4(\lambda_1^{\pm})^2+2\lambda_1^{\pm}\lambda_2^{\pm}\\
&\le6(\lambda_1^{\pm})^2
\end{align*}
and the claim follows. Combining this claim with the identity
$K^\perp_{min}-\frac{R}{12}=\frac{\lambda^{+}_1+\lambda^{-}_1}{2}$
(see, for instance, identity (2.5) in \cite{CoRi}),  we obtain
\begin{lemma}\label{detWup2}
Let $(M,g)$ be a four-dimensional manifold. Then
\[
\frac{1}{\sqrt{6}}\left(|W^+|+|W^-|\right)\le\frac{R}{6}-2K^\perp_{min}.
\]
\end{lemma}
\subsection{Basic facts of shrinker}
In this subsection, we review several known results for shrinkers,
with particular attention to the four-dimensional setting. From \eqref{Eq1},
we derive the following identities:
\begin{eqnarray}\label{listshr}
&R+\Delta f=n \lambda, \label{tra}\\
&\nabla_i R=2R_{ij} \nabla_j f, \label{tracric}\\
& \Delta_f R=2\lambda R-2|Ric|^2, \label{Sequat}\\
& \Delta_fR_{ij}=R_{ij}-2R_{ikjl}R_{kl}, \label{evo}
\end{eqnarray}
where $\Delta_f:=\Delta-\nabla_{\nabla f}$ is the $f$-Laplacian
and $R_{ij}$ is the Ricci curvature of $(M,g,f)$.

On a complete shrinker $(M,g,f)$, Chen \cite{[Chen]} showed that the scalar
curvature satisfies $R\ge 0$. Pigola, Rimoldi and Setti \cite{PRS} further showed
that $R>0$ unless $(M,g,f)$ is the Gaussian shrinker. Let $\rd=Ric-\frac{R}{4}g$.
Using the above equations, a direct computation yields
\begin{lemma}\label{l-eqric}
Let $(M,g,f)$ be a four-dimensional gradient Ricci soliton. Then
\[
\frac{1}{2}\Delta_f|\rd|^2=|\nabla\rd|^2+ 2\lambda |\rd|^2-2R_{ijkl}{\R}_{ik}{\R}_{jl}
-\frac{1}{2}R|\rd|^2,
\]
where $\rd$ is the traceless Ricci curvature.
\end{lemma}

For a closed four-dimensional gradient Ricci soliton, integrating
\eqref{Sequat} by parts yields
\begin{align*}
\int_M|\nabla R|^2dv &=-\int_MR\Delta Rdv\\
&=-\int_MR\Big(\langle\nabla f,\nabla R\rangle+2\lambda R-2|Ric|^2\Big)dv\\
&=\frac 12\int_MR^2\Delta f dv-2\lambda\int_M R^2dv+\frac 12\int_M R^3dv+2\int_M R|\rd|^2dv.
\end{align*}
Combining this identity with \eqref{tra} immediately gives
\begin{lemma}\label{l-id}
Let $(M,g,f)$ be a closed four-dimensional gradient Ricci soliton. Then
\[
\int_M|\nabla R|^2dv=2\int_M R|\rd|^2dv.
\]
\end{lemma}

Next, we recall a Weitzenb\"ock formula and an integral identity derived in \cite{CaTr16}.
These formulas play an important role in the proof of Theorem \ref{Rigid2}.
\begin{proposition}\label{proWe}
Let $(M,g,f)$ be a four-dimensional shrinker satisfying \eqref{Eq1}. Then
\begin{equation}\label{Weitz}
\Delta_f|W^\pm|^2=2|\nabla W^\pm|^2+4\lambda|W^\pm|^2-36\,\mathrm{det}W^{\pm}
-\langle Ric\circ Ric, W^\pm\rangle.
\end{equation}
Furthermore, if $(M,g,f)$ is closed, then
\[
\int_M\langle Ric\circ Ric, W^{\pm}\rangle dv=4\int_M |\delta W^{\pm}|^2 dv,
\]
where $\circ$ denotes the Kulkarni-Nomizu product.
\end{proposition}

Finally, we give a simple observation concerning four-dimensional Einstein manifolds:
a lower bound on the $2$nd-Ricci curvature yields an upper bound for the sectional
curvature. This fact will be used in the proof of Theorem \ref{Rig1}.
\begin{proposition}\label{relat}
Let $(M,g)$ be a four-dimensional Einstein manifold satisfying
\[
Ric=\lambda\, g
\]
for some constant $\lambda>0$. If $Ric_2(M)\ge \frac{1}{6}\,\lambda$, then
every sectional curvature satisfies $K\le\frac{5}{6}\lambda$.
\end{proposition}
\begin{proof}[Proof of Proposition \ref{relat}]
For any point $x\in M$, let $P\subset T_xM$ be a plane realizing the maximal
sectional curvature at $x$. We may select an orthonormal basis $\{e_i\}^4_{i=1}$ of
$T_xM$ such that $P=\mathrm{Span}\{e_1,e_2\}$ and
\[
Ric(e_1,e_1)=K(e_1, e_2)+K(e_1, e_3)+K(e_1, e_4)=\lambda.
\]
Suppose $Ric_2(M)\ge\frac{1}{6}\lambda$ at $x\in M$. It follows that
\[
K(e_1, e_3)+K(e_1, e_4)\ge\frac{1}{6}\lambda,
\]
and hence $K(e_1, e_2)\le\frac{5}{6}\lambda$. That is, $K\le\frac{5}{6}\lambda$
at $x\in M$. Thus, the result follows.
\end{proof}

\section{Pinching estimates for $2$nd-Ricci curvature}\label{prorig1}
In this section, motivated by Proposition 2.1 from reference \cite{[Cat]},
we first establish a series of algebraic curvature estimates under a lower
bound condition on the $2$nd-Ricci curvature. These results generalize the
the conclusions of \cite{[Cat]}.

\begin{hidden}
\subsection{NEW}
\begin{lemma}\label{lem:3A}
Let $(\overline K_P)_{P\in\mathcal P}$ be real numbers satisfying
\eqref{eq:cone2}, and let $(c_P)_{P\in\mathcal P}$ be nonnegative numbers
satisfying
\begin{equation}\label{eq:coeff}
  \sum_{Q\ne P,\ Q\ne P^{c}} c_Q\ \ge\ c_P+c_{P^{c}}
  \qquad\text{for every } P\in\mathcal P .
\end{equation}
Then
\[
  \sum_{P\in\mathcal P}c_P\,\overline K_P\ \ge\ 0 .
\]
\end{lemma}
 
\begin{proof}
Put $\mathcal N:=\{P\in\mathcal P:\overline K_P<0\}$. If $P,Q\in\mathcal N$ are
distinct with $Q\ne P^{c}$, then $\overline K_P+\overline K_Q<0$, contradicting
\eqref{eq:cone2}. Hence any two distinct elements of $\mathcal N$ are opposite,
so $\mathcal N$ is either empty or contained in a single opposite couple
$\{P_0,P_0^{c}\}$ from \eqref{eq:couples}.
 
If $\mathcal N=\emptyset$, the conclusion is immediate since $c_P\ge0$.
Otherwise fix such a couple and set
\[
  m:=\max\bigl\{-\overline K_{P_0},\ -\overline K_{P_0^{c}},\ 0\bigr\}\ >\ 0 .
\]
By definition $\overline K_P\ge-m$ for $P\in\{P_0,P_0^{c}\}$. If
$Q\notin\{P_0,P_0^{c}\}$, then $Q$ meets both $P_0$ and $P_0^{c}$, so
\eqref{eq:cone2} gives $\overline K_Q\ge-\overline K_{P_0}$ and
$\overline K_Q\ge-\overline K_{P_0^{c}}$; moreover $Q\notin\mathcal N$ forces
$\overline K_Q\ge0$. Therefore $\overline K_Q\ge m$ for every such $Q$.
Using $c_P\ge0$ throughout,
\[
  \sum_{P\in\mathcal P}c_P\overline K_P
  =c_{P_0}\overline K_{P_0}+c_{P_0^{c}}\overline K_{P_0^{c}}
   +\sum_{Q\ne P_0,P_0^{c}}c_Q\overline K_Q
  \ \ge\ -m\bigl(c_{P_0}+c_{P_0^{c}}\bigr)
   +m\sum_{Q\ne P_0,P_0^{c}}c_Q\ \ge\ 0 ,
\]
the last inequality being \eqref{eq:coeff} applied to $P=P_0$.
\end{proof}
 
\begin{lemma}\label{lem:3B}
Let $\lambda_1\le\lambda_2\le\lambda_3\le\lambda_4$ be real numbers and set
$d_P:=(\lambda_i-\lambda_j)^2$ for $P=\{i,j\}\in\mathcal P$. Then both families
\[
  c_P:=d_P
  \qquad\text{and}\qquad
  c_P:=d_{P^{c}}
\]
are nonnegative and satisfy \eqref{eq:coeff}.
\end{lemma}
 
\begin{proof}
Write $\alpha:=\lambda_2-\lambda_1$, $\beta:=\lambda_3-\lambda_2$ and
$\gamma:=\lambda_4-\lambda_3$, so that $\alpha,\beta,\gamma\ge0$ and
\[
  d_{12}=\alpha^2,\quad d_{23}=\beta^2,\quad d_{34}=\gamma^2,\quad
  d_{13}=(\alpha+\beta)^2,\quad d_{24}=(\beta+\gamma)^2,\quad
  d_{14}=(\alpha+\beta+\gamma)^2 .
\]
Expanding, one verifies the three identities
\begin{align}
  d_{12}+d_{23}+d_{34}+d_{14}-d_{13}-d_{24}&=(\alpha+\gamma)^2,
     \label{eq:id1}\\
  d_{12}+d_{13}+d_{24}+d_{34}-d_{14}-d_{23}&=(\alpha-\gamma)^2,
     \label{eq:id2}\\
  d_{13}+d_{14}+d_{23}+d_{24}-d_{12}-d_{34}&=(\alpha+2\beta+\gamma)^2 .
     \label{eq:id3}
\end{align}
These are exactly the three instances of \eqref{eq:coeff} for the family
$c_P=d_P$, one for each opposite couple in \eqref{eq:couples}, and every
right-hand side is nonnegative.
 
For the second family, observe that $P\mapsto P^{c}$ is an involution of
$\mathcal P$ preserving each opposite couple in \eqref{eq:couples} setwise.
Consequently, for every $P\in\mathcal P$,
\[
  \sum_{Q\ne P,\ Q\ne P^{c}}d_{Q^{c}}
  =\sum_{Q'\ne P^{c},\ Q'\ne P}d_{Q'} ,
  \qquad
  d_{P^{c}}+d_{(P^{c})^{c}}=d_{P}+d_{P^{c}} ,
\]
so that \eqref{eq:coeff} for the family $(d_{P^{c}})_{P}$ is the very same
system of three inequalities \eqref{eq:id1}--\eqref{eq:id3}.
\end{proof}
 
\begin{remark}
Only \eqref{eq:id1} and \eqref{eq:id2} can degenerate: they vanish precisely
when $\lambda_2-\lambda_1=-(\lambda_4-\lambda_3)$, resp.\
$\lambda_2-\lambda_1=\lambda_4-\lambda_3$. Since $\alpha,\gamma\ge0$, the first
forces $\alpha=\gamma=0$.
\end{remark}
 \subsection{The curvature estimates}
 
\begin{lemma}\label{lem:31}
Let $(M,g)$ be a four-dimensional manifold. If the 2nd-Ricci curvature satisfies
$\operatorname{Ric}_2(M)\ge2\varepsilon R$ for some $\varepsilon\in\mathbb R$, then
\[
  R_{ijkl}\Rc_{ik}\Rc_{jl}
  \ \le\ \frac{1-16\varepsilon}{4}\,R\,|\Rico|^2
   +\Rc_{ij}\Rc_{ik}\Rc_{jk} .
\]
\end{lemma}
 
\begin{proof}
With the notation \eqref{eq:eigen}--\eqref{eq:modified-data}, and using
$\Rc_{ij}=\lambda_i\delta_{ij}$ together with
$\overline\mu_k=\sum_{i\ne k}\overline K_{ik}$, we compute
\begin{equation}\label{eq:31}
\begin{aligned}
  \overline R_{ijkl}\Rc_{ik}\Rc_{jl}-\overline R_{ij}\Rc_{ik}\Rc_{jk}
  &=\sum_{i,j=1}^{4}\lambda_i\lambda_j\overline K_{ij}
    -\sum_{k=1}^{4}\overline\mu_k\lambda_k^2\\
  &=2\sum_{i<j}\lambda_i\lambda_j\overline K_{ij}
    -\sum_{i<j}\bigl(\lambda_i^2+\lambda_j^2\bigr)\overline K_{ij}
  =-\sum_{i<j}\bigl(\lambda_i-\lambda_j\bigr)^2\overline K_{ij}.
\end{aligned}
\end{equation}
By Lemma \ref{lem:3B} the coefficients $c_{\{i,j\}}=(\lambda_i-\lambda_j)^2$
satisfy \eqref{eq:coeff}, and by hypothesis $(\overline K_P)$ satisfies
\eqref{eq:cone2}. Lemma \ref{lem:3A} therefore gives
\[
  I:=\sum_{i<j}(\lambda_i-\lambda_j)^2\,\overline K_{ij}\ \ge\ 0 ,
\]
and hence, by \eqref{eq:31},
\begin{equation}\label{eq:31bar}
  \overline R_{ijkl}\Rc_{ik}\Rc_{jl}\ \le\ \overline R_{ij}\Rc_{ik}\Rc_{jk}.
\end{equation}
 
It remains to rewrite \eqref{eq:31bar} in terms of $Rm$. Since
$\sum_{i\ne j}\lambda_i\lambda_j=\bigl(\sum_i\lambda_i\bigr)^2-\sum_i\lambda_i^2
=-|\Rico|^2$, formulas \eqref{eq:modified-data} give
\[
  \overline R_{ijkl}\Rc_{ik}\Rc_{jl}
   =R_{ijkl}\Rc_{ik}\Rc_{jl}+\varepsilon R\,|\Rico|^2,
  \qquad
  \overline R_{ij}\Rc_{ik}\Rc_{jk}
   =R_{ij}\Rc_{ik}\Rc_{jk}-3\varepsilon R\,|\Rico|^2 .
\]
Combining these with
$R_{ij}\Rc_{ik}\Rc_{jk}=\Rc_{ij}\Rc_{ik}\Rc_{jk}+\tfrac14R|\Rico|^2$, inequality
\eqref{eq:31bar} becomes
\[
  R_{ijkl}\Rc_{ik}\Rc_{jl}+\varepsilon R|\Rico|^2
  \ \le\ \Rc_{ij}\Rc_{ik}\Rc_{jk}+\frac{1-12\varepsilon}{4}R|\Rico|^2 ,
\]
which is the desired estimate.
\end{proof}
 
\begin{lemma}\label{lem:32}
Let $(M,g)$ be a four-dimensional manifold. If the 2nd-Ricci curvature satisfies
$\operatorname{Ric}_2(M)\ge2\varepsilon R$ for some $\varepsilon\in\mathbb R$, then
\[
  R_{ijkl}\Rc_{ik}\Rc_{jl}
  \ \le\ \frac{1-16\varepsilon}{4}\,R\,|\Rico|^2
   -3\,\Rc_{ij}\Rc_{ik}\Rc_{jk} .
\]
\end{lemma}
 
\begin{proof}
We keep the notation of the previous proof. First,
\begin{equation}\label{eq:33}
  \overline R_{ikjl}\Rc_{ij}\Rc_{kl}-\frac{\overline R}{4}|\Rico|^2
  =\sum_{i,j=1}^{4}\lambda_i\lambda_j\overline K_{ij}
   -\frac{\overline R}{4}\sum_{k=1}^{4}\lambda_k^2 ,
\end{equation}
and since
$\overline R=\sum_{i,j}\overline K_{ij}=2\sum_{i<j}\overline K_{ij}$, the
right-hand side of \eqref{eq:33} equals
\begin{equation}\label{eq:34}
  2\sum_{i<j}\lambda_i\lambda_j\overline K_{ij}
  -\frac12\sum_{i<j}\overline K_{ij}\sum_{k=1}^{4}\lambda_k^2
  =\sum_{i<j}\Bigl(2\lambda_i\lambda_j-\frac12\sum_{k}\lambda_k^2\Bigr)
     \overline K_{ij} .
\end{equation}
 
Set
\[
  J:=\sum_{i<j}\Bigl(\sum_{k}\lambda_k^2
      -\frac32\bigl(\lambda_i^2+\lambda_j^2\bigr)-\lambda_i\lambda_j\Bigr)
      \overline K_{ij} .
\]
Writing $\{\bar i,\bar j\}=\{1,2,3,4\}\setminus\{i,j\}$ with $\bar i<\bar j$ and
using $\sum_k\lambda_k=0$, we have for each $i<j$
\[
  \sum_{k}\lambda_k^2-\frac32(\lambda_i^2+\lambda_j^2)-\lambda_i\lambda_j
  =\lambda_{\bar i}^2+\lambda_{\bar j}^2-\frac12(\lambda_i+\lambda_j)^2
  =\lambda_{\bar i}^2+\lambda_{\bar j}^2
     -\frac12(\lambda_{\bar i}+\lambda_{\bar j})^2
  =\frac12(\lambda_{\bar i}-\lambda_{\bar j})^2 ,
\]
so that
\begin{equation}\label{eq:J}
  2J=\sum_{i<j}(\lambda_{\bar i}-\lambda_{\bar j})^2\,\overline K_{ij}
    =\sum_{P\in\mathcal P}d_{P^{c}}\,\overline K_P .
\end{equation}
By Lemma \ref{lem:3B} the family $c_P=d_{P^{c}}$ satisfies \eqref{eq:coeff};
Lemma \ref{lem:3A} therefore yields
\[
  J\ \ge\ 0 .
\]
 
We now conclude. Adding $\tfrac34\sum_{i<j}(\lambda_i-\lambda_j)^2\overline K_{ij}$
to \eqref{eq:34} and using
\[
  2\lambda_i\lambda_j-\frac12\sum_k\lambda_k^2
   +\frac34(\lambda_i-\lambda_j)^2
  =-\frac12\Bigl(\sum_k\lambda_k^2-\frac32(\lambda_i^2+\lambda_j^2)
     -\lambda_i\lambda_j\Bigr),
\]
we obtain from $J\ge0$ that
\[
  \sum_{i<j}\Bigl(2\lambda_i\lambda_j-\frac12\sum_k\lambda_k^2\Bigr)
     \overline K_{ij}
  =-\frac12 J-\frac34\sum_{i<j}(\lambda_i-\lambda_j)^2\overline K_{ij}
  \ \le\ -\frac34\sum_{i<j}(\lambda_i-\lambda_j)^2\overline K_{ij} .
\]
Combining this with \eqref{eq:33}, \eqref{eq:34} and identity \eqref{eq:31},
\[
  \overline R_{ikjl}\Rc_{ij}\Rc_{kl}-\frac{\overline R}{4}|\Rico|^2
  \ \le\ \frac34\Bigl(\overline R_{ijkl}\Rc_{ik}\Rc_{jl}
       -\overline R_{ij}\Rc_{ik}\Rc_{jk}\Bigr),
\]
which, after rearranging, is equivalent to
\[
  \overline R_{ikjl}\Rc_{ij}\Rc_{kl}
  \ \le\ \overline R\,|\Rico|^2-3\,\overline R_{ij}\Rc_{ik}\Rc_{jk} .
\]
Substituting \eqref{eq:modified-data} exactly as at the end of the proof of
Lemma \ref{lem:31} recovers the stated estimate.
\end{proof}
 
\begin{corollary}\label{cor:33}
Let $(M,g)$ be a four-dimensional manifold. If the 2nd-Ricci curvature satisfies
$\operatorname{Ric}_2(M)\ge2\varepsilon R$ for some $\varepsilon\in\mathbb R$, then
\[
  R_{ijkl}\Rc_{ik}\Rc_{jl}
  \ \le\ \frac{1-16\varepsilon}{4}\,R\,|\Rico|^2
   +(4s-3)\,\Rc_{ij}\Rc_{ik}\Rc_{jk}
\]
for every $s\in[0,1]$.
\end{corollary}
 
\begin{proof}
Multiply the estimate of Lemma \ref{lem:31} by $s$, the estimate of Lemma
\ref{lem:32} by $1-s$, and add; the coefficient of
$\Rc_{ij}\Rc_{ik}\Rc_{jk}$ is $s\cdot1+(1-s)\cdot(-3)=4s-3$.
\end{proof}
 
  \subsection{Original}
\end{hidden}

\begin{lemma}\label{pest}
Let $(M,g)$ be a four-dimensional manifold. If the $2$nd-Ricci curvature
satisfies $Ric_2(M)\ge2\varepsilon R$ for some $\varepsilon\in\mathbb{R}$,
then
\[
R_{ijkl}{\R}_{ik}{\R}_{jl}\le\frac{1-16\varepsilon}{4}R|\rd|^2+\R_{ij}\R_{ik}\R_{jk}.
\]
\end{lemma}

\begin{proof}[Proof of Lemma \ref{pest}]
Let $\{e_i\}$, $i=1,\ldots, 4$, be the eigenvectors of $\rd$ and let $\lambda_i$ be
the corresponding eigenvalues with $\lambda_1\le\lambda_2\le\lambda_3\le\lambda_4$.
Let $K_{ij}$ denote the sectional curvature of the two-plane spanned by $e_i$ and
$e_j$. Similar to the argument of \cite{[Cat]}, given the lower bound assumption
$Ric_2(M)\ge2\varepsilon R$, we introduce the modified curvature tensor
\[
\overline{Rm}=Rm-\frac{\varepsilon}{2}R\, g\circ g.
\]
In particular, its associated Ricci tensor, scalar curvature and sectional curvature satisfy
\[
\overline{Ric}=Ric-3\varepsilon R g\,,\quad\overline{R}=\big(1-12\varepsilon\big) R \quad\hbox{and}\quad \overline{K}_{ij}=K_{ij}-\varepsilon R.
\]
Furthermore, let $\mu_k$ and $\overline{\mu}_k$ denote the eigenvalues of $Ric$
and $\overline{Ric}$ corresponding to the common eigenvector $e_k$, respectively.
One has
\[
\mu_k=\sum_{i\neq k}K_{ik}\quad \mathrm{and}\quad
\overline{\mu}_k=\sum_{i\neq k} \overline{K}_{ik}.
\]
Denoting by $\overline{R}_{ijkl}$ the components of $\overline{Rm}$, we compute
\begin{equation}
\begin{aligned}\label{equno}
\overline{R}_{ijkl}\R_{ik}\R_{jl}-\overline{R}_{ij}\R_{ik}\R_{jk}&=\sum_{i,j=1}^4\lambda_{i}\lambda_j\overline{K}_{ij}-\sum_{k=1}^4\mu_k\lambda_k^2\\
&=2\sum_{i<j}\lambda_i\lambda_j\overline{K}_{ij}-\sum_{i<j}
\big(\lambda_i^2+\lambda_j^2\big)\overline{K}_{ij}\\
&=-\sum_{i<j}\big(\lambda_i-\lambda_j\big)^2\overline{K}_{ij}.
\end{aligned}
\end{equation}
We now establish the following \emph{claim}: if $Ric_2(M)\ge2\varepsilon R$, then
\[
I:=\sum_{i<j}\left(\lambda_i-\lambda_j\right)^2\overline{K}_{ij}\ge0.
\]
Indeed, the condition $Ric_2(M)\ge2\varepsilon R$ implies
\begin{equation}\label{crit}
\overline{K}_{ij}+\overline{K}_{ik}=K_{ij}+K_{ik}-2\varepsilon R\ge 0
\end{equation}
for any triple of indices $\{e_i, e_j, e_k\}$. In the following, we will
apply inequality \eqref{crit} to verify the claim.

Case 1: Suppose $\overline{K}_{12}\le0$. Then $\overline{K}_{13}, \overline{K}_{14}, \overline{K}_{23},\overline{K}_{24}\ge0$, while $\overline{K}_{34}\ge0$ or
$\overline{K}_{34}\le0$.

(a). If $\overline{K}_{12}\le0$, $\overline{K}_{13}, \overline{K}_{14}, \overline{K}_{23},\overline{K}_{24}\ge0$ and $\overline{K}_{34}\ge0$,
it is straightforward to check $I\ge0$.
Indeed, using $\lambda_1\le\lambda_2\le\lambda_3\le\lambda_4$ and \eqref{crit}, we compute
\begin{align*}
I:&=(\lambda_1-\lambda_2)^2\overline{K}_{12}+(\lambda_1-\lambda_3)^2\overline{K}_{13}
+(\lambda_1-\lambda_4)^2\overline{K}_{14}\\
&\quad+(\lambda_2-\lambda_3)^2\overline{K}_{23}+(\lambda_2-\lambda_4)^2\overline{K}_{24}
+(\lambda_3-\lambda_4)^2\overline{K}_{34}\\
&\ge(\lambda_1-\lambda_3)^2\big(\overline{K}_{12}+\overline{K}_{13}\big)
+(\lambda_1-\lambda_4)^2\overline{K}_{14}\\
&\quad+(\lambda_2-\lambda_3)^2\overline{K}_{23}+(\lambda_2-\lambda_4)^2\overline{K}_{24}
+(\lambda_3-\lambda_4)^2\overline{K}_{34}\\
&\ge0.
\end{align*}

(b). If $\overline{K}_{12}\le0$, $\overline{K}_{13}, \overline{K}_{14}, \overline{K}_{23},\overline{K}_{24}\ge0$ and $\overline{K}_{34}\le0$,
we combine $\lambda_1\le\lambda_2\le\lambda_3\le\lambda_4$ and \eqref{crit} to deduce
\begin{align*}
I:&=(\lambda_1-\lambda_2)^2\overline{K}_{12}+(\lambda_1-\lambda_3)^2\overline{K}_{13}
+(\lambda_1-\lambda_4)^2\overline{K}_{14}\\
&\quad+(\lambda_2-\lambda_3)^2\overline{K}_{23}+(\lambda_2-\lambda_4)^2\overline{K}_{24}
+(\lambda_3-\lambda_4)^2\overline{K}_{34}\\
&\ge(\lambda_1-\lambda_3)^2\overline{K}_{12}+(\lambda_1-\lambda_3)^2\overline{K}_{13}
+(\lambda_3-\lambda_4)^2\overline{K}_{14}\\
&\quad+(\lambda_2-\lambda_3)^2\overline{K}_{23}+(\lambda_2-\lambda_4)^2\overline{K}_{24}
+(\lambda_3-\lambda_4)^2\overline{K}_{34}\\
&=(\lambda_1-\lambda_3)^2\big(\overline{K}_{12}+\overline{K}_{13}\big)
+(\lambda_3-\lambda_4)^2\big(\overline{K}_{14}+\overline{K}_{34}\big)\\
&\quad+(\lambda_2-\lambda_3)^2\overline{K}_{23}+(\lambda_2-\lambda_4)^2\overline{K}_{24}\\
&\ge0.
\end{align*}

Case 2: Suppose $\overline{K}_{13}\le0$. Then $\overline{K}_{12}, \overline{K}_{14}, \overline{K}_{23},\overline{K}_{34}\ge0$. The critical subcase occurs when
 $\overline{K}_{24}\le0$; the case $\overline{K}_{24}\ge0$ is trivial.
In this setting, recall that
\begin{align*}
I:&=(\lambda_1-\lambda_2)^2\overline{K}_{12}+(\lambda_1-\lambda_3)^2\overline{K}_{13}
+(\lambda_1-\lambda_4)^2\overline{K}_{14}\\
&\quad+(\lambda_2-\lambda_3)^2\overline{K}_{23}+(\lambda_2-\lambda_4)^2\overline{K}_{24}
+(\lambda_3-\lambda_4)^2\overline{K}_{34}.
\end{align*}
To eliminate $\lambda_1$, applying $\lambda_1\le\lambda_2\le\lambda_3\le\lambda_4$, we introduce the following nonnegative ordered parameters
\[
x=\lambda_2-\lambda_1\ge 0,\quad y=\lambda_3-\lambda_1\ge x,\quad z=\lambda_4-\lambda_1\ge y.
\]
Substituting these into $I$,
\[
I=x^2\overline{K}_{12}+y^2\overline{K}_{13}+z^2\overline{K}_{14}
+(y-x)^2\overline{K}_{23}+(z-x)^2\overline{K}_{24}+(z-y)^2\overline{K}_{34}.
\]
Now we let
\[
m:=\max\left\{-\overline{K}_{13},-\overline{K}_{24}\right\}\ge0.
\]
Then $\overline{K}_{13},\overline{K}_{24}\ge-m$. The assumption
$\overline{K}_{ij}+\overline{K}_{ik}\ge 0$ implies
 $\overline{K}_{12},\overline{K}_{14}, \overline{K}_{23},\overline{K}_{34}\ge m$.
So
\begin{align*}
I&=y^2\overline{K}_{13}+(z-x)^2\overline{K}_{24}+x^2\overline{K}_{12}+z^2\overline{K}_{14}
+(y-x)^2\overline{K}_{23}+(z-y)^2\overline{K}_{34}\\
&\ge\big[-y^2-(z-x)^2+x^2+z^2+(y-x)^2+(z-y)^2\big]m\\
&=(x-y+z)^2m\\
&\ge0.
\end{align*}

Case 3: Suppose $\overline{K}_{14}\le0$. Then $\overline{K}_{12}, \overline{K}_{13}, \overline{K}_{24},\overline{K}_{34}\ge0$. The worst-case scenario is
$\overline{K}_{23}\le0$. For this case, similar to the argument of
Case 2, using $\lambda_1\le\lambda_2\le\lambda_3\le\lambda_4$, we introduce 
nonnegative ordered parameters
\[
x=\lambda_2-\lambda_1\ge 0,\quad y=\lambda_3-\lambda_1\ge x,\quad z=\lambda_4-\lambda_1\ge y.
\]
Let
\[
m:=\max\left\{-\overline{K}_{14},-\overline{K}_{23}\right\}\ge0.
\]
Then $\overline{K}_{14},\overline{K}_{23}\ge-m$. The assumption
$\overline{K}_{ij}+\overline{K}_{ik}\ge 0$ implies
 $\overline{K}_{12},\overline{K}_{13}, \overline{K}_{24},\overline{K}_{34}\ge m$.
Putting these information together, we have
\begin{align*}
I:&=(\lambda_1-\lambda_2)^2\overline{K}_{12}+(\lambda_1-\lambda_3)^2\overline{K}_{13}
+(\lambda_1-\lambda_4)^2\overline{K}_{14}\\
&\quad+(\lambda_2-\lambda_3)^2\overline{K}_{23}+(\lambda_2-\lambda_4)^2\overline{K}_{24}
+(\lambda_3-\lambda_4)^2\overline{K}_{34}\\
&=z^2\overline{K}_{14}+(y-x)^2\overline{K}_{23}+x^2\overline{K}_{12}+y^2\overline{K}_{13} 
+(z-x)^2\overline{K}_{24}+(z-y)^2\overline{K}_{34}\\
&\ge\big[-z^2-(y-x)^2+x^2+y^2+(z-x)^2+(z-y)^2\big]m\\
&=(x-z+y)^2m\\
&\ge0.
\end{align*}

Combining all cases above, we conclude that $I\ge 0$ whenever $Ric_2(M)\ge2\varepsilon R$,
which completes the proof of the claim.

By the definitions of $\overline{Rm}$ and $\overline{Ric}$, we rearrange the inequality
to get
\begin{align*}
R_{ijkl}\R_{ik}\R_{jl}+\varepsilon R|\rd|^2&\le
R_{ij}\R_{ik}\R_{jk}-3\varepsilon R|\rd|^2\\
&=\R_{ij}\R_{ik}\R_{jk}+\frac{1-12\varepsilon}{4} R|\rd|^2,
\end{align*}
which proves the desired estimate.
\end{proof}

Motivated by the second part of Proposition 2.1 in \cite{[Cat]}, we derive
another algebraic curvature estimate.
\begin{lemma}\label{pest2}
Let $(M,g)$ be a four-dimensional manifold. If the $2$nd-Ricci curvature
satisfies $Ric_2(M)\ge2\varepsilon R$ for some $\varepsilon\in\mathbb{R}$,
then
\[
R_{ijkl}{\R}_{ik}{\R}_{jl}\le\frac{1-16\varepsilon}{4}R|\rd|^2-3\R_{ij}\R_{ik}\R_{jk}.
\]
\end{lemma}
\begin{proof}[Proof of Lemma \ref{pest2}]
We adopt the same notation introduced in the proof of Lemma \ref{pest}.
One readily verifies that
\begin{equation}\label{curiden}
\overline{R}_{ikjl}\R_{ij}\R_{kl}-\frac{\overline{R}}{4}|\rd|^2
=\sum_{i,j=1}^4\lambda_i\lambda_j\overline{K}_{ij}-\frac{\overline{R}}{4}\sum_{k=1}^4\lambda_k^2.
\end{equation}
Recall the modified scalar curvature $\overline{R}$ admits the expression
\[
\overline{R}=g^{ik}g^{jl}\overline{R}_{ijkl}=\sum^4_{i,j=1}\overline{K}_{ij}=2\sum_{i<j}\overline{K}_{ij}.
\]
Substituting this identity yields
\begin{equation}
\begin{aligned}\label{imprest}
\sum_{i,j=1}^4\lambda_i\lambda_j\overline{K}_{ij}-\frac{\overline{R}}{4}\sum_{k=1}^4\lambda_k^2
&=2\sum_{i<j}\lambda_i\lambda_j\overline{K}_{ij}-\frac{1}{2}\sum_{i<j}\overline{K}_{ij}\sum_{k=1}^4\lambda_k^2\\
&=\sum_{i<j}\left(2\lambda_i\lambda_j-\frac{1}{2}\sum_{k=1}^4\lambda_k^2\right)\overline{K}_{ij}.
\end{aligned}
\end{equation}
We now establish the following \emph{claim}: if $Ric_2(M)\ge2\varepsilon R$, then
\[
J:=\sum_{k=1}^4\lambda_k^2\cdot\sum_{i<j}\overline{K}_{ij}
-\sum_{i<j}\left[\frac{3}{2}(\lambda_i^2+\lambda_j^2)+\lambda_j\lambda_j\right]\overline{K}_{ij}\ge0.
\]

Indeed, we utilize the identity $\sum_{k=1}^4\lambda_k=0$ to rearrange
the expression for $J$:
\begin{equation*}
\begin{aligned}
J&=\sum_{i<j}\left[\sum_{k=1}^4\lambda_k^2-\frac{3}{2}(\lambda_i^2+\lambda_j^2)
-\lambda_i\lambda_j\right]\overline{K}_{ij}\\
&=\sum_{i<j}\left[(\lambda_{\bar{i}}^2+\lambda_{\bar{j}}^2)-\frac{1}{2}(\lambda_i^2+\lambda_j^2)
-\lambda_i\lambda_j\right]\overline{K}_{ij}\\
&=\sum_{i<j}\left[(\lambda_{\bar{i}}^2+\lambda_{\bar{j}}^2)-\frac{1}{2}(\lambda_i+\lambda_j)^2
\right]\overline{K}_{ij}\\
&=\sum_{i<j}\left[(\lambda_{\bar{i}}^2+\lambda_{\bar{j}}^2)-\frac{1}{2}(\lambda_{\bar{i}}+\lambda_{\bar{j}})^2
\right]\overline{K}_{ij}\\
&=\frac{1}{2}\sum_{i<j}(\lambda_{\bar{i}}-\lambda_{\bar{j}})^2\overline{K}_{ij},
\end{aligned}
\end{equation*}
where $\bar{i},\bar{j}\in\{1,2,3,4\}\backslash \{i,j\}$ and $\bar{i}<\bar{j}$.
Since the condition $Ric_2(M)\ge2\varepsilon R$ implies inequality \eqref{crit},
we apply \eqref{crit} and repeat the three-case argument from the proof of Lemma \ref{pest} to verify the claim.

Case 1: Suppose $\overline{K}_{12}\le0$. Then $\overline{K}_{13}, \overline{K}_{14}, \overline{K}_{23},\overline{K}_{24}\ge0$, while $\overline{K}_{34}\ge0$ or
$\overline{K}_{34}\le0$.

(a). If $\overline{K}_{12}\le0$, $\overline{K}_{13}, \overline{K}_{14}, \overline{K}_{23},\overline{K}_{24}\ge0$ and $\overline{K}_{34}\ge0$, we can deduce $J\ge0$.
Indeed, using $\lambda_1\le\lambda_2\le\lambda_3\le\lambda_4$ and \eqref{crit}, we compute
\begin{align*}
2J:&=(\lambda_3-\lambda_4)^2\overline{K}_{12}+(\lambda_2-\lambda_4)^2\overline{K}_{13}
+(\lambda_2-\lambda_3)^2\overline{K}_{14}\\
&\quad+(\lambda_1-\lambda_4)^2\overline{K}_{23}+(\lambda_1-\lambda_3)^2\overline{K}_{24}
+(\lambda_1-\lambda_2)^2\overline{K}_{34}\\
&\ge(\lambda_3-\lambda_4)^2(\overline{K}_{12}+\overline{K}_{13})
+(\lambda_2-\lambda_3)^2\overline{K}_{14}\\
&\quad+(\lambda_1-\lambda_4)^2\overline{K}_{23}+(\lambda_1-\lambda_3)^2\overline{K}_{24}
+(\lambda_1-\lambda_2)^2\overline{K}_{34}\\
&\ge0.
\end{align*}

(b). If $\overline{K}_{12}\le0$, $\overline{K}_{13}, \overline{K}_{14}, \overline{K}_{23},\overline{K}_{24}\ge0$ and $\overline{K}_{34}\le0$,
combining $\lambda_1\le\lambda_2\le\lambda_3\le\lambda_4$ and \eqref{crit} gives
\begin{align*}
2J:&=(\lambda_3-\lambda_4)^2\overline{K}_{12}+(\lambda_2-\lambda_4)^2\overline{K}_{13}
+(\lambda_2-\lambda_3)^2\overline{K}_{14}\\
&\quad+(\lambda_1-\lambda_4)^2\overline{K}_{23}+(\lambda_1-\lambda_3)^2\overline{K}_{24}
+(\lambda_1-\lambda_2)^2\overline{K}_{34}\\
&\ge(\lambda_3-\lambda_4)^2(\overline{K}_{12}+\overline{K}_{13})
+(\lambda_2-\lambda_3)^2\overline{K}_{14}\\
&\quad+(\lambda_1-\lambda_4)^2\overline{K}_{23}
+(\lambda_1-\lambda_2)^2(\overline{K}_{24}+\overline{K}_{34})\\
&\ge0.
\end{align*}

Case 2: Suppose $\overline{K}_{13}\le0$. Then $\overline{K}_{12}, \overline{K}_{14}, \overline{K}_{23},\overline{K}_{34}\ge0$, while $\overline{K}_{24}\ge0$ or
$\overline{K}_{24}\le0$. We only analyze the critical subcase $\overline{K}_{24}\le0$.
Similar to the argument of Lemma \ref{pest}: Case 2, applying $\lambda_1\le\lambda_2\le\lambda_3\le\lambda_4$, we introduce the 
following nonnegative ordered parameters
\[
x=\lambda_2-\lambda_1\ge 0,\quad y=\lambda_3-\lambda_1\ge x,\quad z=\lambda_4-\lambda_1\ge y.
\]
Let
\[
m:=\max\left\{-\overline{K}_{13},-\overline{K}_{24}\right\}\ge0.
\]
Then $\overline{K}_{13},\overline{K}_{24}\ge-m$. The assumption
$\overline{K}_{ij}+\overline{K}_{ik}\ge 0$ implies
 $\overline{K}_{12},\overline{K}_{14}, \overline{K}_{23},\overline{K}_{34}\ge m$.
Putting these information together, we have
\begin{align*}
2J:&=(\lambda_3-\lambda_4)^2\overline{K}_{12}+(\lambda_2-\lambda_4)^2\overline{K}_{13}
+(\lambda_2-\lambda_3)^2\overline{K}_{14}\\
&\quad+(\lambda_1-\lambda_4)^2\overline{K}_{23}+(\lambda_1-\lambda_3)^2\overline{K}_{24}
+(\lambda_1-\lambda_2)^2\overline{K}_{34}\\
&=(z-y)^2\overline{K}_{12}+(z-x)^2\overline{K}_{13}
+(y-x)^2\overline{K}_{14}+z^2\overline{K}_{23}+y^2\overline{K}_{24}+x^2\overline{K}_{34}\\
&\ge\big[-(z-x)^2-y^2+(z-y)^2+(y-x)^2+z^2+x^2\big]m\\
&=(x-y+z)^2m\\
&\ge0.
\end{align*}

Case 3: Suppose $\overline{K}_{14}\le0$. Then $\overline{K}_{12}, \overline{K}_{13}, \overline{K}_{24},\overline{K}_{34}\ge0$. The worst-case scenario is $\overline{K}_{23}\le0$.
In this setting, similar to the above argument, we still have $2J\ge0$.

Combining all cases above, we conclude $J\ge 0$ whenever $Ric_2(M)\ge2\varepsilon R$,
which completes the proof of the claim.

We now proceed to complete the proof of our result. Using the claim established above,
we further estimate expression \eqref{imprest} as
\begin{equation*}
\begin{aligned}
\sum_{i,j=1}^n\lambda_i\lambda_j\overline{K}_{ij}-\frac{\overline{R}}{4}\sum_{k=1}^4\lambda_k^2
\le-\frac 34\sum_{i<j}(\lambda_i-\lambda_j)^2\overline{K}_{ij}.
\end{aligned}
\end{equation*}
Substituting this into identity \eqref{equno} yields
\[
\sum_{i,j=1}^n\lambda_i\lambda_j\overline{K}_{ij}-\frac{\overline{R}}{4}\sum_{k=1}^4\lambda_k^2
\le\frac 34\left(\overline{R}_{ijkl}\R_{ik}\R_{jl}-\overline{R}_{ij}\R_{ik}\R_{jk}\right).
\]
Combining this estimate with \eqref{curiden}, we have
\[
\overline{R}_{ikjl}\R_{ij}\R_{kl}-\frac{\overline{R}}{4}|\rd|^2\le\frac 34\left(\overline{R}_{ijkl}\R_{ik}\R_{jl}-\overline{R}_{ij}\R_{ik}\R_{jk}\right).
\]
Equivalently,
\[
\overline{R}_{ikjl}\R_{ij}\R_{kl}\le\overline{R}|\rd|^2-3\overline{R}_{ij}\R_{ik}\R_{jk}.
\]
Substituting the definitions of $\overline{Rm}$, $\overline{Ric}$
and $\overline{R}$ into the above inequality recovers the desired estimate.
\end{proof}

Taking a convex linear combination of the bounds from Lemma \ref{pest} and Lemma \ref{pest2}
immediately produces the following uniform estimate, which will be essential for the proof of
Theorem \ref{Rig1}.
\begin{corollary}\label{con-com}
Let $(M,g)$ be a four-dimensional manifold. If the $2$nd-Ricci curvature
satisfies $Ric_2(M)\ge2\varepsilon R$ for some $\varepsilon\in\mathbb{R}$,
then
\[
R_{ijkl}{\R}_{ik}{\R}_{jl}\le\frac{1-16\varepsilon}{4}R|\rd|^2-(3-4s)\R_{ij}\R_{ik}\R_{jk}
\]
for every $s\in [0,1]$.
\end{corollary}

We now apply Corollary \ref{con-com} to prove Theorem \ref{Rig1} by following
the argument of \cite{[Cat]}.
\begin{proof}[Proof of Theorem \ref{Rig1}]
From \cite{[Chen],{PRS}}, we know that either $R\equiv 0$ or $R>0$. If $R\equiv 0$,
the result of \cite{PRS} guarantees that $(M,g,f)$ is isometric to $\RR^4$
with the standard Euclidean metric. For the remainder of the proof, we assume
$R>0$ on $M$. The condition $Ric_2(M)\ge2\varepsilon R>0$ then forces $M^4$
to be closed by \cite{LN} (see also \cite{WW}).

We integrate the evolution identity from Lemma \ref{l-eqric}
over the closed shrinker $(M,g,f)$ and get that
\begin{equation}\label{ident}
\begin{aligned}
0&=\frac{1}{2}\int_M\langle \nabla |\rd|^2,\nabla f\rangle dv+\int_M|\nabla\rd|^2 dv
+2\lambda\int_M |\rd|^2dv\\
&\quad-2\int_M R_{ijkl}\R_{ik}\R_{jl} dv-\frac{1}{2}\int_M R|\rd|^2 dv\\
&=\int_M\left(-\frac{1}{2}\Delta f+2\lambda-\frac{R}{2}\right)|\rd|^2dv
+\int_M|\nabla\rd|^2 dv-2\int_M R_{ijkl}\R_{ik}\R_{jl} dv\\
&=\int_M|\nabla \rd|^2dv-2\int_M R_{ijkl}\R_{ik}\R_{jl} dv,
\end{aligned}
\end{equation}
where in the last equality we used the identity $R+\Delta f=4\lambda$ on
four-dimensional shrinkers.

On the other hand, set
\[
F_{ijk}:=\nabla_k{\R}_{ij}+\nabla_j{\R}_{ik}+\nabla_i{\R}_{jk}
-\frac{1}{12}\left(\nabla_k R g_{ij}+\nabla_j R g_{ik}+\nabla_i R g_{jk}\right).
\]
Using the contracted Bianchi identity $\nabla_i\R_{ij}=\frac 14\nabla_j R$, we directly
compute that
\[
|F|^2=3|\nabla \rd|^2+ 6\nabla_k\R_{ij}\nabla_j\R_{ik}-\frac{1}{8}|\nabla R|^2.
\]
Integrating both sides of this identity over $M$ yields
\begin{equation}\label{idtint}
\begin{aligned}
\int_M|\nabla \rd|^2dv&=\frac{1}{3}\int_M|F|^2dv-2\int_M\nabla_k \R_{ij}\nabla_j\R_{ik}dv
+\frac{1}{24}\int_M|\nabla R|^2dv\\
&=\frac{1}{3}\int_M|F|^2dv-2\int_M\nabla_k \R_{ij}\nabla_j\R_{ik}dv
+\frac{1}{12}\int_MR|\rd|^2dv,
\end{aligned}
\end{equation}
where we used Lemma \ref{l-id} in the second equality.
Applying integration by parts, commutation of covariant derivatives and the Bianchi relation $\nabla_i\R_{ij}=\frac{1}{4}\nabla_j R$, we rewrite the cross term
\begin{equation*}
\begin{aligned}
-\int_M\nabla_k\R_{ij}\nabla_j\R_{ik}dv&=\int_M\R_{ij} \nabla_k\nabla_j\R_{ik} dv\\
&=\int_M\Big(\R_{ij}\nabla_j\nabla_k \R_{ik}+R_{kjil}\R_{ij}\R_{kl}+R_{ij}\R_{ik}\R_{jk}\Big)dv \\
&=\int_M\left(\frac{1}{4}\R_{ij}\nabla_i\nabla_j R-R_{ijkl}\R_{ik}\R_{jl}+\R_{ij}\R_{ik}\R_{jk}+ \frac{1}{4} R|\rd|^2\right) dv \\
&=\int_M\left(-\frac{1}{16}|\nabla R|^2-R_{ijkl}\R_{ik}\R_{jl}+\R_{ij}\R_{ik}\R_{jk}+\frac{1}{4} R|\rd|^2\right) dv \\
&=\int_M\left(-R_{ijkl}\R_{ik}\R_{jl}+\R_{ij}\R_{ik}\R_{jk}+\frac{1}{8}R|\rd|^2\right) dv,
\end{aligned}
\end{equation*}
where we used Lemma \ref{l-id} once more in the last equality. Substituting this into
\eqref{idtint}, we get
\[
\int_M|\nabla \rd|^2dv=\frac{1}{3}\int_M|F|^2dv-2\int_MR_{ijkl}\R_{ik}\R_{jl}dv
+2\int_M{\R}_{ij}{\R}_{ik}{\R}_{jk}dv+\frac 13\int_MR|{\overset{\circ}Ric}|^2dv.
\]
Then substituting this into \eqref{ident} gives that
\[
0=\frac{1}{3}\int_M|F|^{2}dv-4\int_MR_{ijkl}{\R}_{ik}{\R}_{jl}dv
+2\int_M{\R}_{ij}{\R}_{ik}{\R}_{jk}dv+\frac 13\int_MR|{\overset{\circ}Ric}|^2dv.
\]
Now recall the condition $Ric_2(M)\ge2\varepsilon R$. Corollary \ref{con-com} gives
that
\[
-4R_{ijkl}{\R}_{ik}{\R}_{jl}\ge(16\varepsilon-1) R|\rd|^2+(12-16s)\R_{ij}\R_{ik}\R_{jk}
\]
for every $s\in [0,1]$. Inserting this inequality yields
\begin{equation}\label{FRic} 
0\ge\frac{1}{3}\int_M|F|^2dv+(14-16s)\int_M{\R}_{ij}{\R}_{ik}{\R}_{jk}dv
+\left(16\varepsilon-\frac 23\right)\int_MR|\rd|^2dv.    
\end{equation}
Taking $s=7/8$, \eqref{FRic} simplifies to
\[
0\ge\frac{1}{3}\int_M|F|^2dv+\left(16\varepsilon-\frac 23\right)\int_MR|\rd|^2dv.
\]
Since $\varepsilon>\frac{1}{24}$, then $16\varepsilon-\frac 23>0$ and we have $\rd=0$.
So $(M,g,f)$ is an Einstein manifold. For such an Einstein shrinker,
the lower bound $Ric_2(M)\ge2\varepsilon R$ for some $\varepsilon>\frac{1}{24}$
implies $Ric_2(M)>\frac{1}{12}R$. This satisfies the pinching condition
\eqref{rigcond}, and we conclude that $(M,g, f)$ is isometric to either
$\SS^4$ or $\mathbb{CP}^2$ equipped with its standard metric, as discussed
in the introduction.
\end{proof}

\begin{remark}
For the critical case $\varepsilon=\frac1{24}$,
\eqref{FRic} implies 
\begin{equation}
	\begin{aligned}\label{equacase}
		0&\ge\frac{1}{3}\int_M|F|^2dv+(14-16s)\int_M{\R}_{ij}{\R}_{ik}{\R}_{jk}dv
		+\left(16\varepsilon-\frac 23\right)\int_MR|\rd|^2dv\\
		&=\frac{1}{3}\int_M|F|^2dv+(14-16s)\int_M{\R}_{ij}{\R}_{ik}{\R}_{jk}dv 
	\end{aligned}
\end{equation}	 
 for every $s\in [0,1]$.  Since $s\in [0,1]$ is arbitrary, 
 then $F\equiv 0$ and $\int_M{\R}_{ij}{\R}_{ik}{\R}_{jk}dv\equiv 0$. This
forces \eqref{equacase} to become equality. Thus, the inequality of Corollary \ref{con-com} becomes equality. Namely,
\[
R_{ijkl}{\R}_{ik}{\R}_{jl}\equiv\frac{R}{12} |\rd|^2,\quad {\R}_{ij}{\R}_{ik}{\R}_{jk}\equiv 0.
\]
Moreover the equality of Corollary \ref{con-com} implies $I\equiv J\equiv 0$ in the proof of Lemmas \ref{pest} and \ref{pest2}. But to the authors, these equalties seem not to imply $\rd=0$.
\end{remark}


\section{Integral pinching estimates for curvature}\label{prorig2}
In this section, we first establish a key lemma, which will be used in the proof
of Theorem \ref{Rigid2}.
\begin{lemma}\label{recond}
Let $(M,g)$ be a closed four-dimensional shrinker, and let $\lambda_1(M)$
be its first eigenvalue of the Laplacian on functions on $M$. If one of the
following conditions holds:
\begin{eqnarray}\label{listcondK}
& K^\perp\ge\frac{(1-16k^2)R^2+8kR\lambda_1(M)}{24[(1-4k)R+\lambda_1(M)]},\label{tiji1}\\
&K^\perp_{min}\ge\frac{(1-16k^2)R+8k\lambda_1(M)}{4(1-4k)(1-2k)R+2(3-8k)\lambda_1(M)}K^\perp_{max},\label{tiji2}\\
& Ric\ge a \quad \text{and} \quad K^\perp\le\frac{2(1-4k)(1-2k)+1-\tfrac{8}{3}k}{4(1-3k)}a \quad \text{for some constant}\,\,\, a>0,\label{tiji3}
\end{eqnarray}
where $0\le k\le1/4$, then
\[
K^\perp\le\frac{2(1-4k)(1-2k)R^2+(3-8k)\lambda_1(M)R}{12[(1-4k)R+\lambda_1(M)]}.
\]
\end{lemma}
\begin{proof}[Proof of Lemma \ref{recond}]
Fix an orthonormal basis $\{e_i\}^4_{i=1}$ for $T_xM$, and let $P_1, P_2, P_3$
be the $2$-planes spanned by $\{e_1, e_2\}$, $\{e_1, e_3\}$, $\{e_1, e_4\}$, respectively.
Recall the identity
\[
K^\perp(P_1)+K^\perp(P_2)+K^\perp(P_3)=\frac{R}{4}.
\]
First assume condition \eqref{tiji1} holds. We rearrange the sum identity to estimate
$K^\perp(P_3)$:
\begin{align*}
K^\perp(P_3) &\le \frac{R}{4}-2\cdot\frac{(1-16k^2)R^2+8kR\lambda_1(M)}{24[(1-4k)R+\lambda_1(M)]} \\
&=\frac{2(1-4k)(1-2k)R^2+(3-8k)\lambda_1(M)R}{12[(1-4k)R+\lambda_1(M)]},
\end{align*}
which yields the desired upper bound.

Next, suppose  \eqref{tiji2} is satisfied. From the plane-sum identity we obtain
\[
\frac{R}{4}\ge K^\perp(P_3)+2\cdot\frac{(1-16k^2)R+8k\lambda_1(M)}{4(1-4k)(1-2k)R+2(3-8k)\lambda_1(M)} K^\perp(P_3).
\]
Solving $K^\perp(P_3)$ gives that
\[
K^\perp(P_3)\le\frac{2(1-4k)(1-2k)R^2+(3-8k)\lambda_1(M)R}{24[(1-4k)R+\lambda_1(M)]},
\]
and the conclusion follows immediately.

Finally, assume $Ric\ge a$. Then we have $R\ge 4a$. Moreover, Lichnerowicz's theorem
(see \cite{lic}) gives the spectral lower bound $\lambda_1(M)\ge\frac{4}{3}a$.
Define the function
\[
F(R,\lambda_1(M)):=\frac{2(1-4k)(1-2k)R^2+(3-8k)\lambda_1(M)R}{12[(1-4k)R+\lambda_1(M)]}.
\]
Then it is nondecreasing in both parameters $R$ and $\lambda_1(M)$.
Substituting minimal values $R=4a$ and $\lambda_1(M)=\frac 43a$,
we compute
\begin{align*}
\frac{2(1-4k)(1-2k)R^2+(3-8k)\lambda_1(M)R}{12[(1-4k)R+\lambda_1(M)]}
&\ge\frac{2(1-4k)(1-2k)(4a)^2+(3-8k)\frac{4}{3}a\cdot4a}{12[(1-4k)4a+\frac{4}{3}a]}\\
&=\frac{2(1-4k)(1-2k)+1-\tfrac{8}{3}k}{4(1-3k)}a.
\end{align*}
Combined with condition \eqref{tiji3}, this implies the desired estimate.
\end{proof}

We now proceed to prove Theorem \ref{Rigid2} by following the argument of
\cite{CaTr22}.

\begin{proof}[Proof of Theorem \ref{Rigid2}]
If the closed shrinker $(M,g,f)$ is half conformally flat, the desired
conclusion follows directly from the classification result of Chen–Wang \cite{CWy}.
We now suppose the contrary and derive a contradiction via the following argument.
Since $g$ is analytic, the quantities $|W^\pm|^2$ are real analytic functions on $M$.
Moreover, if $W^\pm\not\equiv 0$, the zero set
\[
\Omega:=\left\{x\in M\,\,|\,\,|W^+|(x)=0\,\, \text{or}\,\,|W^-|(x)=0\right\}
\]
is finite.

Suppose $(M,g,f)$ is not half conformally flat. Then $\int_M |W^+|dv>0$
and $\int_M |W^-|dv>0$, so there exists a constant $t>0$ such that
\begin{equation}\label{relaW}
\int_M(|W^+|-t|W^-|)dv=0.
\end{equation}
We apply the definition of the $f$-Laplacian, integration by parts and the
basic shrinker identity $R+\Delta f=4\lambda$ to compute that
\begin{align*}
0&=\frac 12\int_M\Delta\left(|W^+|^2+t^2|W^-|^2\right)dv\\
&=\frac 12\int_M\Delta_f\left(|W^+|^2+t^2|W^-|^2\right)dv
+\frac 12\int_M\nabla f\cdot\nabla\left(|W^+|^2+t^2|W^-|^2\right)dv\\
&=\frac 12\int_M\Delta_f\left(|W^+|^2+t^2|W^-|^2\right)dv
-\frac 12\int_M\Delta f \cdot\left(|W^+|^2+t^2|W^-|^2\right)dv\\
&=\frac 12\int_M\Delta_f\left(|W^+|^2+t^2|W^-|^2\right)dv
+\frac 12\int_M(R-4\lambda)\left(|W^+|^2+t^2|W^-|^2\right)dv.
\end{align*}
Applying the Weitzenb\"ock formula \eqref{Weitz} and the Kato inequality
$|\nabla W^\pm|\ge|\nabla |W^\pm||$, the above equality becomes
\begin{align*}
0&\ge\int_M\left(|\nabla|W^+||^2+2\lambda|W^+|^2-18\,\mathrm{det}W^+
-\frac 12\langle Ric\circ Ric, W^+\rangle\right) dv\\
&\quad+t^2\int_M\left(|\nabla|W^-||^2+2\lambda|W^-|^2-18\,\mathrm{det}W^{-}
-\frac 12\langle Ric\circ Ric, W^-\rangle\right) dv\\
&\quad+\frac 12\int_M(R-4\lambda) \left(|W^+|^2+t^2|W^-|^2\right) dv.
\end{align*}
Rearranging terms yields
\begin{align*}
0&\ge\frac 12\int_M\left(\big|\nabla(|W^+|+t|W^-|)\big|^2+\big|\nabla(|W^+|-t|W^-|)\big|^2\right)dv\\
&\quad+\int_M\left(2\lambda|W^+|^2-18\,\mathrm{det}W^+
-\frac 12\langle Ric\circ Ric, W^+\rangle\right)dv\\
&\quad+t^2\int_M\left(2\lambda|W^-|^2-18\,\mathrm{det}W^{-}
-\frac 12\langle Ric\circ Ric, W^-\rangle\right)dv\\
&\quad+\frac 12\int_M(R-4\lambda) \left(|W^+|^2+t^2|W^-|^2\right)dv.
\end{align*}
Let $\lambda_1(M)$ denote the first positive eigenvalue of the Laplacian
on functions. By the Raleigh quotient characterization of
$\lambda_1(M)$ and identity \eqref{relaW}, we have
\[
\int_M\big|\nabla(|W^+|-t|W^-|)\big|^2dv\ge\lambda_1(M)\int_M\Big(|W^+|-t|W^-|\Big)^2dv.
\]
We insert this estimate and discard the nonnegative gradient term $\int_M\big|\nabla(|W^+|+t|W^-|)\big|^2dv$
to deduce that
\begin{align*}
0&\ge\frac{\lambda_1(M)}{2}\int_M\Big(|W^+|-t|W^-|\Big)^2dv
+\int_M\left(2\lambda|W^+|^2-18\,\mathrm{det}W^+-\frac 12\langle Ric\circ Ric, W^+\rangle\right)dv\\
&\quad+t^2\int_M\left(2\lambda|W^-|^2-18\,\mathrm{det}W^{-}-\frac 12\langle Ric\circ Ric, W^-\rangle\right)dv\\
&\quad+\frac 12\int_M(R-4\lambda) \left(|W^+|^2+t^2|W^-|^2\right)dv.
\end{align*}
Applying Lemma \ref{detWup} and the second equality of Proposition \ref{Weitz}, i.e.,
\[
-18\,\mathrm{det}W^{\pm}\ge-3\lambda^{\pm}_3|W^\pm|^2
\quad\text{and}\quad
\frac 12\int_M\langle Ric\circ Ric, W^{\pm}\rangle dv=2\int_M |\delta W^{\pm}|^2 dv,
\]
the above inequality becomes
\begin{align*}
0&\ge\frac{\lambda_1(M)}{2}\int_M\Big(|W^+|-t|W^-|\Big)^2dv
+\int_M\Big(2\lambda|W^+|^2-3\lambda^{+}_3|W^+|^2-2|\delta W^{+}|^2\Big)dv\\
&\quad+t^2\int_M\Big(2\lambda|W^-|^2-3\lambda^{-}_3|W^-|^2-2|\delta W^{-}|^2\Big)dv\\
&\quad+\frac 12\int_M(R-4\lambda) \left(|W^+|^2+t^2|W^-|^2\right)dv.
\end{align*}
Expanding and grouping terms by powers of $t$, we obtain
\begin{align*}
0&\ge t^2\int_M \left[|W^-|^2\left(\frac{\lambda_1(M)}{2}
+\frac{R}{2}-3\lambda^{-}_3\right)-2|\delta W^-|^2\right]dv-t\lambda_1(M)\int_M|W^+||W^-|dv\\
&\quad+\int_M\left[|W^+|^2\left(\frac{\lambda_1(M)}{2}
+\frac{R}{2}-3\lambda^{+}_3\right)-2|\delta W^+|^2\right]dv.
\end{align*}
Using the theorem's integral pinching assumption
$\int_M|\delta W^{\pm}|^2dv\le k\int_M R|W^{\pm}|^2dv$,
where $0\le k\le1/4$, we further get
\begin{equation}
\begin{aligned}\label{quadr}
0&\ge t^2\int_M |W^-|^2\left(\frac{\lambda_1(M)}{2}
+\frac{R}{2}-3\lambda^{-}_3-2kR\right)dv-t\lambda_1(M)\int_M|W^+||W^-|dv\\
&\quad+\int_M|W^+|^2\left(\frac{\lambda_1(M)}{2}+\frac{R}{2}-3\lambda^{+}_3-2kR\right)dv.
\end{aligned}
\end{equation}
The integrand inside \eqref{quadr} is a quadratic polynomial in the real variable $t$.
Define the pointwise integrand
\begin{align*}
\mathcal{P}(t)&:=t^2 |W^-|^2\left(\frac{\lambda_1(M)}{2}
+\frac{R}{2}-3\lambda^{-}_3-2kR\right)-t\lambda_1(M) |W^+||W^-|\\
&\quad+|W^+|^2\left(\frac{\lambda_1(M)}{2}+\frac{R}{2}-3\lambda^{+}_3-2kR\right).
\end{align*}
We will prove $\mathcal{P}(t)\geq 0$ on $M$.

First, we use the first inequality of Lemma \ref{ineqlem}:
$-3\lambda_3^{\pm}\ge\frac R2-6K^\perp_{max}$. This gives a uniform lower bound
for the leading coefficient of $\mathcal{P}(t)$:
\begin{align*}
\frac{\lambda_1(M)}{2}+\frac{R}{2}-3\lambda^{-}_3-2kR\ge\frac{\lambda_1(M)}{2}+(1-2k)R-6K^\perp_{max}>0,
\end{align*}
where in the second inequality we used the fact
\begin{align*}
K^\perp_{max}:=\frac{(16k^2-12k+2)R^2+(3-8k)\lambda_1(M)R}{12[(1-4k)R+\lambda_1(M)]}
<\frac{\lambda_1(M)}{12}+\frac{1-2k}{6}R.
\end{align*}
The discriminant of the quadratic $\mathcal{P}(t)$ reads
\begin{align*}
D&=|W^+|^2|W^-|^2\Big[\lambda^2_1(M)-\big(\lambda_1(M)+(1-4k)R-6\lambda^{-}_3\big)
\big(\lambda_1(M)+(1-4k)R-6\lambda^{+}_3\big)\Big]\\
&=-|W^+|^2|W^-|^2\Big[\big((1-4k)R-6\lambda^{-}_3\big)\big((1-4k)R-6\lambda^{+}_3\big)\\
&\qquad\qquad\qquad\qquad+\lambda_1(M)\big(2(1-4k)R-6(\lambda^{-}_3+\lambda^{+}_3)\big)\Big].
\end{align*}
Our goal is to show $D\le 0$. For brevity, let
\[
\widetilde{D}:=\big[(1-4k)R-6\lambda^{-}_3\big]\big[(1-4k)R-6\lambda^{+}_3\big],
+\lambda_1(M)\big[2(1-4k)R-6(\lambda^{-}_3+\lambda^{+}_3)\big]
\]
so proving $D\le 0$ is equivalent to proving $\widetilde{D}\ge0$.
Clearly, $\widetilde{D}$ can be written as
\[
\widetilde{D}=\big[(1-4k)R-3\lambda^{-}_3-3\lambda^{+}_3\big]^2-9(\lambda^{+}_3-\lambda^{-}_3)^2
+2\lambda_1(M)\big[(1-4k)R-3(\lambda^{-}_3+\lambda^{+}_3)\big].
\]
To verify non-negativity of $\widetilde{D}$, we apply Lemma \ref{ineqlem} and the definition of $K^\perp_{max}$ to obtain two pointwise estimates:
\[
(1-4k)R-3\lambda^{-}_3-3\lambda^{+}_3\ge\left(\frac 32-4k\right)R-6K^\perp_{max}\ge0
\]
and
\[
|\lambda^{+}_3-\lambda^{-}_3|\le2\left(K^\perp_{max}-\frac{R}{12}\right)\ge 0,
\]
where $0\le k\le1/4$. Substituting the above two estimates into the expression for
$\widetilde{D}$ yields
\begin{align*}
\widetilde{D}&\ge\left[\left(\frac 32-4k\right)R-6K^\perp_{max}\right]^2-36\left(K^\perp_{max}-\frac{R}{12}\right)^2
+2\lambda_1(M)\left[\left(\frac 32-4k\right)R-6K^\perp_{max}\right]\\
&=0,
\end{align*}
where we used
\[
K^\perp_{max}=\frac{(16k^2-12k+2)R^2+(3-8k)\lambda_1(M)R}{12[(1-4k)R+\lambda_1(M)]}.
\]
We therefore conclude
\[
D=-|W^+|^2|W^-|^2\widetilde{D}\le0.
\]
Since the quadratic $\mathcal{P}(t)$ has a positive leading coefficient and non-positive discriminant, we deduce $\mathcal{P}(t)\ge0$ pointwise at every $x\in M$ for all real $t$. Combining this with \eqref{quadr}, we show that $\mathcal{P}(t)$ must be zero at each point and all inequalities become equalities. In particular,
by Lemma \ref{ineqlem}, we have $|W^+|\cdot|W^-|\equiv0$ in $M$. Combining
this with the preceding finite set $\Omega$ yields a contradiction about
our assumption of not half conformally flat. This proves Theorem \ref{Rigid2}
under the pinching condition \eqref{la}.

The remaining cases \eqref{lb}, \eqref{lc}
and \eqref{1d} follow directly by applying Lemma \ref{recond}.
\end{proof}

\section{Weighted integral pinching curvature}\label{discuss}

In this section, we prove another rigidity theorem for shrinkers,
namely Theorem \ref{Rigid3}, under a weighted integral curvature pinching
condition. Compared with Theorem \ref{Rigid2}, the key distinction is that
all integral computations here are performed with respect to the weighted
measure $e^{-f}dv$.

\begin{proof}[Proof of Theorem \ref{Rigid3}]
Similar to the proceeding argument of Theorem \ref{Rigid2}, if the closed shrinker
is half conformally flat, by Chen-Wang's classification \cite{CWy}, the conclusion
follows. If not, we can get a contradiction. Suppose that $(M,g,f)$ is not half
conformally flat. Then
\[
\int_M |W^+|e^{-f}dv>0\quad \text{and} \quad\int_M |W^-|e^{-f}dv>0.
\]
Thus there exists a constant $t>0$ such that
\begin{equation}\label{relaW2}
\int_M\left(|W^+|-t|W^-|\right)e^{-f}dv=0.
\end{equation}
Applying stokes' theorem, the Weitzenb\"ock formula \eqref{Weitz} and
the Kato inequality $|\nabla W^\pm|\ge|\nabla |W^\pm||$, we compute that
\begin{align*}
0&=\frac 12\int_M\Delta_f\left(|W^+|^2+t^2|W^-|^2\right)e^{-f}dv\\
&\ge\int_M\left(|\nabla|W^+||^2+2\lambda|W^+|^2-18\,\mathrm{det}W^+
-\frac 12\langle Ric\circ Ric, W^+\rangle\right)e^{-f}dv\\
&\quad+t^2\int_M\left(|\nabla|W^-||^2+2\lambda|W^-|^2-18\,\mathrm{det}W^{-}
-\frac 12\langle Ric\circ Ric, W^-\rangle\right)e^{-f}dv.
\end{align*}
Rearranging terms yields
\begin{equation}\label{WKredu}
\begin{aligned}
0&\ge\frac 12\int_M\left(\big|\nabla(|W^+|+t|W^-|)\big|^2+\big|\nabla(|W^+|-t|W^-|)\big|^2\right)e^{-f}dv\\
&\quad+\int_M\left(2\lambda|W^+|^2-18\,\mathrm{det}W^+
-\frac 12\langle Ric\circ Ric, W^+\rangle\right)e^{-f}dv\\
&\quad+t^2\int_M\left(2\lambda|W^-|^2-18\,\mathrm{det}W^{-}
-\frac 12\langle Ric\circ Ric, W^-\rangle\right)e^{-f}dv.
\end{aligned}
\end{equation}
Let $\lambda_{1,f}(M)$ denote the first positive eigenvalue of the $f$-Laplacian
on functions. By the variation characterization of $\lambda_{1,f}(M)$,
combined with identity \eqref{relaW2}, we have
\[
\int_M\big|\nabla(|W^+|-t|W^-|)\big|^2e^{-f}dv\ge\lambda_{1,f}(M)\int_M\Big(|W^+|-t|W^-|\Big)^2e^{-f}dv.
\]
Substituting this estimate into \eqref{WKredu} and discarding the non-negative
gradient term $\int_M\big|\nabla(|W^+|+t|W^-|)\big|^2e^{-f}dv$, we obtain
\begin{align*}
0&\ge\frac{\lambda_{1,f}(M)}{2}\int_M\Big(|W^+|-t|W^-|\Big)^2e^{-f}dv\\
&\quad+\int_M\left(2\lambda|W^+|^2-18\,\mathrm{det}W^+-\frac 12\langle Ric\circ Ric, W^+\rangle\right)e^{-f}dv\\
&\quad+t^2\int_M\left(2\lambda|W^-|^2-18\,\mathrm{det}W^{-}-\frac 12\langle Ric\circ Ric, W^-\rangle\right)e^{-f}dv.
\end{align*}
From Lemma \ref{detWup} and the theorem's assumption \eqref{piweyl}, we have
two key facts:
\[
-18\,\mathrm{det}W^{\pm}\ge-3\lambda^{\pm}_3|W^\pm|^2
\quad\text{and}\quad
-\frac 12\int_M\langle Ric\circ Ric, W^{\pm}\rangle e^{-f}dv\ge2k\int_M|W^{\pm}|^2e^{-f}dv.
\]
Putting them into the above inequality, we get
\begin{align*}
0&\ge\frac{\lambda_{1,f}(M)}{2}\int_M\Big(|W^+|-t|W^-|\Big)^2e^{-f}dv
+\int_M\Big(2\lambda|W^+|^2-3\lambda^{+}_3|W^+|^2-2k|W^{+}|^2\Big)e^{-f}dv\\
&\quad+t^2\int_M\Big(2\lambda|W^-|^2-3\lambda^{-}_3|W^-|^2-2k|W^{-}|^2\Big)e^{-f}dv.
\end{align*}
Rearranging terms by powers of $t$, we arrive at
\begin{equation}
\begin{aligned}\label{quadr2}
0&\ge t^2\int_M|W^-|^2\left(\frac{\lambda_{1,f}(M)}{2}
+2\lambda-3\lambda^{-}_3-2k\right)e^{-f}dv-t\lambda_{1,f}(M)\int_M|W^+||W^-|e^{-f}dv\\
&\quad+\int_M|W^+|^2\left(\frac{\lambda_{1,f}(M)}{2}
+2\lambda-3\lambda^{+}_3-2k\right)e^{-f}dv.
\end{aligned}
\end{equation}
We now analyze the pointwise integrand appearing inside \eqref{quadr2}.
Define the quadratic function
\begin{align*}
\mathcal{Q}(t)&:=t^2 |W^-|^2\left(\frac{\lambda_{1,f}(M)}{2}
+2\lambda-3\lambda^{-}_3-2k\right)-t\lambda_{1,f}(M) |W^+||W^-|\\
&\quad+|W^+|^2\left(\frac{\lambda_{1,f}(M)}{2}+2\lambda-3\lambda^{+}_3-2k\right).
\end{align*}
We claim $\mathcal{Q}(t)\ge0$ pointwise everywhere on $M$.

The expression $\mathcal{Q}(t)$ is quadratic in the real variable $t$.
First, applying the first inequality of Lemma \ref{ineqlem}:
$-3\lambda_3^{\pm}\ge\frac R2-6K^\perp_{max}$, we have a uniform positive
lower bound for the leading coefficient of $\mathcal{Q}(t)$:
\begin{align*}
\frac{\lambda_{1,f}(M)}{2}+2\lambda-3\lambda^{-}_3-2k
\ge\frac{\lambda_{1,f}(M)}{2}+2(\lambda-k)+\frac R2-6K^\perp_{max}>0,
\end{align*}
where in the second inequality we used the fact
\begin{align*}
K^\perp_{max}:=\frac{(4\lambda-4k)(4\lambda-4k+R)
+\lambda_{1,f}(M)(8\lambda-8k+R)}{12(4\lambda-4k+\lambda_{1,f}(M))}
<\frac{\lambda_{1,f}(M)+R}{12}+\frac{\lambda-k}{3}.
\end{align*}
Next, we compute the discriminant of $\mathcal{Q}(t)$:
\begin{align*}
D&=|W^+|^2|W^-|^2\Big[\lambda^2_{1,f}(M)-\big(\lambda_{1,f}(M)
+4\lambda-4k-6\lambda^{-}_3\big)\big(\lambda_{1,f}(M)+4\lambda-4k-6\lambda^{+}_3\big)\Big]\\
&=-|W^+|^2|W^-|^2\Big[\big(4\lambda-4k-6\lambda^{-}_3\big)\big(4\lambda-4k-6\lambda^{+}_3\big)
+\lambda_{1,f}(M)\big(8\lambda-8k-6(\lambda^{-}_3+\lambda^{+}_3)\big)\Big].
\end{align*}
We will prove $D\le 0$. For simplicity, we set
\[
\widetilde{D}:=\big(4\lambda-4k-6\lambda^{-}_3\big)\big(4\lambda-4k-6\lambda^{+}_3\big)
+\lambda_{1,f}(M)\big(8\lambda-8k-6(\lambda^{-}_3+\lambda^{+}_3)\big)
\]
and our goal reduces to showing $\widetilde{D}\ge0$. Notice that $\widetilde{D}$ can be written as
\begin{equation}\label{Ddelta}
\widetilde{D}=\big[4\lambda-4k-3\lambda^{-}_3-3\lambda^{+}_3\big]^2-9(\lambda^{+}_3-\lambda^{-}_3)^2
+2\lambda_{1,f}(M)\big[4\lambda-4k-3(\lambda^{-}_3+\lambda^{+}_3)\big].
\end{equation}
To verify $\widetilde{D}\ge0$, we apply Lemma \ref{ineqlem} and the definition $K^\perp_{max}$
to get
\[
4\lambda-4k-3\lambda^{-}_3-3\lambda^{+}_3\ge4\lambda-4k+\frac R2-6K^\perp_{max}\ge0
\]
and
\[
|\lambda^{+}_3-\lambda^{-}_3|\le2\left(K^\perp_{max}-\frac{R}{12}\right)\ge 0,
\]
where $0\le k\le1/4$. Using two estimates, we have a lower bound of \eqref{Ddelta}
\begin{align*}
\widetilde{D}&\ge\left(4\lambda-4k+\frac R2-6K^\perp_{max}\right)^2-36\left(K^\perp_{max}-\frac{R}{12}\right)^2
+2\lambda_{1,f}(M)\left(4\lambda-4k+\frac R2-6K^\perp_{max}\right)\\
&=0,
\end{align*}
where in the second equality we used
\[
K^\perp_{max}=\frac{(4\lambda-4k)(4\lambda-4k+R)+\lambda_{1,f}(M)(8\lambda-8k+R)}{12(4\lambda-4k+\lambda_{1,f}(M))}.
\]
Therefore we conclude
\[
D=-|W^+|^2|W^-|^2\widetilde{D}\le0,
\]
which confirms the claim that $\mathcal{Q}(t)\geq 0$ pointwise everywhere on $M$.
Combining this claim with \eqref{quadr2}, we conclude that
the integrand $\mathcal{Q}(t)\equiv 0$ at every point, and all inequalities
above become equalities. In particular, by Lemma \ref{ineqlem}, we get
$|W^+|\cdot|W^-|\equiv0$ in $M$. The rest of the proof of Theorem
\ref{Rigid3} by the similar argument of Theorem \ref{Rigid2}.
\end{proof}

\section{Hitchin-Thorpe type inequality}\label{HTI}
In this section, we prove Theorem \ref{HT} by
a similar argument of \cite{CaJW}.
\begin{proof}[Proof of Theorem \ref{HT}]
For a closed four-dimensional gradient Ricci soliton, integrating identity \eqref{Sequat}
and using \eqref{tra}, we compute
\begin{align*}
\int_M|Ric|^2dv&=\lambda\int_M R dv-\frac 12\int_M\Delta_f Rdv\\
&=4\lambda^2 Vol(M)+\frac 12\int_M\langle\nabla f, \nabla R\rangle dv\\
&=4\lambda^2 Vol(M)-\frac 12\int_MR \Delta f  dv\\
&=4\lambda^2 Vol(M)-\frac 12\int_MR(4\lambda-R)dv\\
&=\int_M \frac{R^2}{2}dv-4\lambda^2 Vol(M),
\end{align*}
where in the last line we used $\int_M Rdv=4\lambda Vol(M)$ from identity \eqref{tra}.
Substituting this into the four-dimensional Gauss-Bonnet-Chern formula for
closed manifold
\[
8\pi^2\chi(M)=\int_M\left(|W^+|^2+|W^-|^2+\frac{R^2}{24}-\frac{1}{2}|\rd|^2\right)dv
\]
and using the decomposition $|\rd|^2=|Ric|^2-\frac{R^2}{4}$, we obtain
\begin{equation}\label{HT1}
8\pi^2\chi(M)=\int_M \left(|W^+|^2+|W^-|^2\right)dv-\frac{1}{12}\int_MR^2dv+2\lambda^2Vol(M).
\end{equation}
Combining this with the classical Hirzebrush's theorem
\[
12\pi^2\tau(M)=\int_M(|W^+|^2-|W^-|^2)dv,
\]
we get that
\begin{equation}\label{CTfor}
2\chi(M)\pm3\tau(M)=\frac{1}{2\pi^2}\int_M|W^\pm|^2dv-\frac{1}{48\pi^2}\int_MR^2dv
+\frac{\lambda^2}{2\pi^2}Vol(M).
\end{equation}
By Theorem 4.1 in \cite{CaTr16}, the pinching assumption \eqref{dWpin} implies
\[
\int_M|W^\pm|^2dv>\frac{4}{11}\pi^2(2\chi(M)\pm3\tau(M)).
\]
Notice that since we assumes that the shrinker is non-trivial, we may assume $W^\pm\not\equiv 0$. 

In the following. we assume  $\varepsilon\lambda\le K^{\perp}\le\sigma\lambda$ for some constants
$\varepsilon\le \sigma\le 1/2$.
Then $K^{\perp}\le\sigma \lambda$ implies $R\le 12\sigma\lambda$.
Substituting these into \eqref{CTfor} and using $\int_M Rdv=4\lambda Vol(M)$
from \eqref{tra}, we get that
\begin{equation}
\begin{aligned}\label{CTfor2a}
\frac{9}{11}\left[2\chi(M)\pm3\tau(M)\right]&>-\frac{1}{48\pi^2}\int_MR^2dv
+\frac{\lambda^2}{2\pi^2}Vol(M)\\
&\ge\frac{1}{2\pi^2}\left(1-2\sigma\right)\lambda^2Vol(M)
\end{aligned}
\end{equation}
where $\sigma\le1/2$.

On the other hand, by Lemma \ref{detWup2}, $K^{\perp}\ge\varepsilon\lambda$ implies that
\[
|W^+|+|W^-|\le\sqrt{6}\left(\frac{R}{6}-2\varepsilon\lambda\right),
\]
and hence
\[
|W^+|^2+|W^-|^2\le 24\varepsilon^2\lambda^2+\frac{R^2}{6}-4\varepsilon\lambda R.
\]
Substituting this into \eqref{HT1}, and using $R\le 12\sigma\lambda$
and $\int_M Rdv=4\lambda Vol(M)$ yields
\begin{align*}
8\pi^2\chi(M)&<\frac{1}{12}\int_M R^2dv+2(12\varepsilon^2-8\varepsilon+1)\lambda^2Vol(M)\\
&\le4\sigma \lambda^2Vol(M)+2(12\varepsilon^2-8\varepsilon+1)\lambda^2Vol(M)\\
&=2(12\varepsilon^2-8\varepsilon+2\sigma+1)\lambda^2Vol(M).
\end{align*}
Note that $2(12\varepsilon^2-8\varepsilon+2\sigma+1)\ge2(12\varepsilon^2-6\varepsilon+1)=24(\varepsilon-\frac14)^2+\frac 12>0$. Plugging this into \eqref{CTfor2a}, we get
\[
\left(\frac{2}{3}-\frac{22}{27}\cdot\frac{1-2\sigma}{12\varepsilon^2-8\varepsilon+2\sigma+1}\right)\chi(M)>|\tau(M)|
\]
provided the coefficient of $\chi(M)$ is positive. Thus the theorem follows.
\end{proof}

\textbf{Acknowledgement}.  J.-Y. Wu is partially supported by NSFC (Grant No. 12571144).

\medskip 

\textbf{AI disclosure}. OpenAI's ChatGPT only assisted the authors in 
checking and correcting Lemmas
\ref{pest} and \ref{pest2}
and Theorem \ref{HT} of this paper.
The authors take full responsibility for the mathematical content.


\end{document}